\documentclass[twoside,11pt]{article}

\usepackage{jmlr2e}

\usepackage{graphicx}

\usepackage[utf8]{inputenc} 
\usepackage[T1]{fontenc}    
\usepackage{lmodern}
\usepackage{hyperref}       
\usepackage{url}            
\usepackage{booktabs}       
\usepackage{amsfonts}       
\usepackage{nicefrac}       
\usepackage{microtype}      

\usepackage{epsfig}
\usepackage{amssymb}
\usepackage{amsmath}
\usepackage{amsfonts}
\usepackage{bbding}
\usepackage{array}
\usepackage{caption,tabularx,booktabs}
\usepackage{arydshln}
\usepackage{xargs}                      
\usepackage[pdftex,dvipsnames]{xcolor}  

\usepackage{mathtools}  
\usepackage{amsmath}
\usepackage{amssymb}
\usepackage{tabulary}

\usepackage{pifont}
\usepackage{algorithm}
\usepackage{algorithmic}
\usepackage{float}
\floatstyle{plain} 
\restylefloat{figure}

\usepackage{hyperref}
\usepackage{breakurl}

\newtheorem{thm}{Theorem}
\newtheorem{lem}{Lemma}

\newtheorem{rem}{Remark}
\newtheorem{ass}{Assumption}

\makeatletter
\newcounter{subthm} 
\let\savedc@thm\c@hyp

\newcommand{\normhyp}{%
  \let\c@hyp\savedc@hyp 
  \renewcommand\thehyp{\arabic{hyp}}%
} 
\makeatother

\makeatletter
\newcounter{subass} 
\let\savedc@ass\c@hyp

\makeatother

\newcommand\tagthis{\addtocounter{equation}{1}\tag{\theequation}}
\DeclareMathOperator{\Prob}{\mathbb{P}}           

\DeclareMathOperator{\Exp}{\mathbb{E}}           
\DeclareMathOperator{\R}{\mathbb{R}} 
\DeclareMathOperator{\Ocal}{\mathcal{O}} 
\newcommand{\eqdef}{\stackrel{\text{def}}{=}}

\newcommand{\ve}[2]{\langle #1, #2 \rangle}

\usepackage{array}
\usepackage{caption,tabularx,booktabs}
\usepackage{arydshln}

\usepackage{xargs}                      
\usepackage[pdftex,dvipsnames]{xcolor}  
\usepackage[colorinlistoftodos,prependcaption,textsize=tiny]{todonotes}
\newcommandx{\unsure}[2][1=]{\todo[inline,linecolor=red,backgroundcolor=red!25,bordercolor=red,#1]{#2}}
\newcommandx{\change}[2][1=]{\todo[linecolor=blue,backgroundcolor=blue!25,bordercolor=blue,#1]{#2}}
\newcommandx{\info}[2][1=]{\todo[linecolor=OliveGreen,inline,backgroundcolor=OliveGreen!25,bordercolor=OliveGreen,#1]{#2}}
\newcommandx{\improvement}[2][1=]{\todo[linecolor=Plum,inline,backgroundcolor=Plum!25,bordercolor=Plum,#1]{#2}}

\newcommand\blfootnote[1]{%
  \begingroup
  \renewcommand\thefootnote{}\footnote{#1}%
  \addtocounter{footnote}{-1}%
  \endgroup
}

\firstpageno{1}

\begin{document}

\title{New Convergence Aspects of Stochastic Gradient Algorithms}

\author{\name Lam M. Nguyen* \email LamNguyen.MLTD@ibm.com \\
       \addr IBM Research, Thomas J. Watson Research Center\\
       Yorktown Heights, NY 10598, USA
       \AND
       \name Phuong Ha Nguyen* \email phuongha.ntu@gmail.com \\
       \addr Department of Electrical and Computer Engineering \\
       University of Connecticut, Storrs, CT 06268, USA
       \AND
	   \name Peter Richt\'{a}rik \email peter.richtarik@kaust.edu.sa \\
       \addr Computer, Electrical and Math.\ Sciences and Engineering Division \\
       King Abdullah University of Science and Technology, Thuwal, KSA       
       \AND
	   \name Katya Scheinberg \email ks2375@cornell.edu \\
	   \addr School of Operations Research and Information Engineering \\
	   Cornell University, Ithaca, NY 14850, USA \\
	   \name Martin Tak\'{a}\v{c} \email Takac.MT@gmail.com \\
       \addr Department of Industrial and Systems Engineering\\
       Lehigh University, Bethlehem, PA 18015, USA  \\
      \name Marten van Dijk \email marten.van$\_$dijk@uconn.edu \\
       \addr Department of Electrical and Computer Engineering \\
       University of Connecticut, Storrs, CT 06268, USA
       }

\editor{N/A}

\maketitle


\begin{abstract}
The classical convergence analysis of SGD is carried out under the assumption that the norm of the stochastic gradient is uniformly bounded. While this might hold for some loss functions, it is violated for cases where the objective function is strongly convex.  In \cite{bottou2016optimization}, a new analysis of convergence of SGD is performed under the assumption that stochastic gradients are bounded with respect to the true gradient norm. We show that for stochastic problems arising in machine learning such bound always holds; and we also propose an alternative convergence analysis of SGD with diminishing learning rate regime. We then move on to the asynchronous parallel setting, and prove convergence of Hogwild! algorithm in the same regime in the case of diminished learning rate. 
  It is well-known that SGD converges if a sequence of learning rates $\{\eta_t\}$ satisfies $\sum_{t=0}^\infty \eta_t \rightarrow \infty$ and $\sum_{t=0}^\infty \eta^2_t < \infty$. We show the convergence of SGD for strongly convex objective function without using bounded gradient assumption when $\{\eta_t\}$ is a diminishing sequence and $\sum_{t=0}^\infty \eta_t \rightarrow \infty$. In other words, we extend the current state-of-the-art class of learning rates satisfying the convergence of SGD.
\end{abstract}


\blfootnote{* Equally contributed. Corresponding author: Lam M. Nguyen}

\section{Introduction}
\label{intro}

We are interested in solving the following stochastic optimization problem
\begin{align*}
\min_{w \in \mathbb{R}^d} \left\{ F(w) = \mathbb{E} [ f(w;\xi) ] \right\}, \tagthis \label{main_prob_expected_risk}  
\end{align*}
where $\xi$ is a random variable  obeying some distribution. 

In the case of empirical risk minimization with a training set $\{(x_i,y_i)\}_{i=1}^n$, $\xi_i$ is a random variable that is defined 
 by a single random sample  $(x,y)$  pulled uniformly from the training set. Then,  by defining  $f_i(w) := f(w;\xi_{i})$, empirical risk minimization reduces to  
\begin{gather}\label{main_prob}
\min_{w \in \mathbb{R}^d} \left\{ F(w) = \frac{1}{n} \sum_{i=1}^n f_i(w) \right\}.  
\end{gather}

 Problem  \eqref{main_prob} arises frequently in supervised learning applications \cite{ESL}. For  a wide range of applications, such as linear regression and logistic regression, the objective function $F$ is strongly convex and each $f_i$, $i \in [n]$, is convex and has Lipschitz continuous gradients (with Lipschitz constant $L$).   Given a training set $\{(x_i,y_i)\}_{i=1}^n$ with $x_i \in\R^{d}, y_i\in\R$, the $\ell_2$-regularized least squares regression model, for example, is written as \eqref{main_prob} with $f_i(w)\eqdef (\langle x_i,w \rangle -y_i)^2 + \frac{\lambda}{2} \|w\|^2$. The $\ell_2$-regularized logistic regression  for binary classification is written with $f_i(w) \eqdef \log (1+\exp(-y_i \langle x_i,w \rangle)) + \frac{\lambda}{2}\|w\|^2$, $y_i\in\{-1,1\}$. It is well established by now that solving this type of problem by gradient descent (GD) \cite{nesterov2004,Nocedal2006NO} may be prohibitively expensive and stochastic gradient descent (SGD) is thus preferable.  Recently, a class of variance reduction methods \cite{SAG,SAGA,SVRG,nguyen2017sarah} has been proposed in order to reduce the computational cost. All these methods explicitly exploit the finite sum form of \eqref{main_prob} and thus they  have some disadvantages for very large scale machine learning problems and are not applicable to \eqref{main_prob_expected_risk}. 

To apply SGD to the general form \eqref{main_prob_expected_risk}  one needs to assume existence of unbiased gradient estimators. This is usually defined as follows: 
 $$\mathbb{E_\xi}[\nabla f(w;\xi)] = \nabla F(w),$$ 
 for any fixed $w$. 
Here we make an important observation: if we view \eqref{main_prob_expected_risk} not as a general stochastic problem but as the expected risk minimization problem, where $\xi$ corresponds to a random data sample pulled from a distribution,
then  \eqref{main_prob_expected_risk} has  an additional key property: for each realization of the random variable $\xi$, $f(w;\xi)$ is a convex function with Lipschitz continuous gradients.
Notice that traditional analysis of SGD for general stochastic problem of the form  \eqref{main_prob_expected_risk} does not make any assumptions on individual function realizations. 
 In this paper we derive convergence properties for SGD applied to \eqref{main_prob_expected_risk} with these additional assumptions on $f(w;\xi)$ and also extend to the case when  $f(w;\xi)$  are not necessarily convex.



Regardless of the properties of $f(w;\xi)$ we assume that $F$ in \eqref{main_prob_expected_risk}  is strongly convex. We define the (unique) optimal solution of $F$ as $w_{*}$.  
\begin{ass}[$\mu$-strongly convex]
\label{ass_stronglyconvex}
The objective function $F: \mathbb{R}^d \to \mathbb{R}$ is a $\mu$-strongly convex, i.e., there exists a constant $\mu > 0$ such that $\forall w,w' \in \mathbb{R}^d$, 
\begin{gather*}
F(w) - F(w') \geq \langle \nabla F(w'), (w - w') \rangle + \frac{\mu}{2}\|w - w'\|^2. \tagthis\label{eq:stronglyconvex_00}
\end{gather*}
\end{ass}
It is well-known in literature \cite{nesterov2004,bottou2016optimization} that Assumption \ref{ass_stronglyconvex} implies
\begin{align*}
2\mu [ F(w) - F(w_{*}) ] \leq \| \nabla F(w) \|^2 \ , \ \forall w \in \mathbb{R}^d. \tagthis\label{eq:stronglyconvex}
\end{align*}

The classical theoretical analysis of SGD assumes that the  \textit{ stochastic gradients are uniformly bounded}, i.e. there exists a finite (fixed) constant $\sigma<\infty$, such that
\begin{align}\label{bounded_grad_ass}
\mathbb{E}[\| \nabla f(w; \xi) \|^2] \leq \sigma^2 \ , \ \forall w \in \mathbb{R}^d
\end{align}
(see e.g. \cite{ShalevShwartz2007,Nemirovski2009,Hogwild,Hazan2014,Rakhlin2012}, etc.). However, this assumption is clearly false if $F$ is strongly convex. Specifically, under this assumption together with strong convexity,  $\forall w \in \mathbb{R}^d$, we have
\begin{align*}
2\mu [ F(w) - F(w_{*}) ] & \overset{\eqref{eq:stronglyconvex}}{\leq} \| \nabla F(w) \|^2 = \left\| \mathbb{E}[ \nabla f(w; \xi) ] \right\|^2 \\ &\leq \mathbb{E}[ \| \nabla f(w; \xi) \|^2 ] \overset{\eqref{bounded_grad_ass}}{\leq} \sigma^2. 
\end{align*} 
Hence, 
\begin{align*}
F(w) \leq \frac{\sigma^2}{2\mu} + F(w_{*}) \ , \ \forall w \in \mathbb{R}^d. 
\end{align*}
On the other hand strong convexity and $\nabla F(w_{*})=0$ imply
\begin{align*}
F(w) \geq {\mu} \|w-w_*\|^2+ F(w_{*}) \ , \ \forall w \in \mathbb{R}^d. 
\end{align*}
The last two inequalities are clearly in  contradiction with each other for sufficiently large  $\|w-w_*\|^2$. 

Let us consider the following example: $f_1(w)= \frac{1}{2} w^2$ and $f_2(w)=w$ with $F(w)=\frac{1}{2}(f_1(w)+f_2(w))$. Note that $F$ is strongly convex, while individual realizations are not necessarily so.  Let $w_0=0$, for any number $t\geq 0$, with probability $\frac{1}{2^t}$ the steps of SGD algorithm for all $i< t$ are $w_{i+1}=w_i-\eta_i$. This implies that
$w_t=-\sum_{i=1}^t \eta_i$ and  since $\sum_{i=1}^\infty\eta_i =\infty$ then $|w_t|$ can be arbitrarily large for large enough $t$ with probability $\frac{1}{2^t}$. Noting that for this example, $\mathbb{E}[\| \nabla f(w_t; \xi) \|^2] =\frac{1}{2}w^2_t+\frac{1}{2}$, we see that $\mathbb{E}[\| \nabla f(w_t; \xi) \|^2]$ can also be arbitrarily large. 





Recently, in the review paper \cite{bottou2016optimization}, convergence of SGD for general stochastic optimization problem was analyzed under the following assumption:  there exist constants $M$ and $N$ such that $\mathbb{E}[\| \nabla f(w_t; \xi_t) \|^2] \leq M \| \nabla F(w_t) \|^2 + N$, where $w_t$, $t \geq 0$, are generated by the SGD algorithm. This assumption does not contradict strong convexity, however, in general, constants $M$ and $N$ are unknown, while $M$ is used to determine the learning rate $\eta_t$, see \cite{bottou2016optimization}.  In addition, the rate of convergence of the SGD algorithm depends  on $M$ and $N$. In this paper we show that under the smoothness assumption on individual realizations  $f(w,\xi)$ it is possible to derive the bound $\mathbb{E}[\| \nabla f(w; \xi) \|^2] \leq M_0 [F(w) - F(w_*)] + N$  with specific values of $M_0$, and $N$ for $\forall w \in \mathbb{R}^d$, which in turn  implies the bound $\mathbb{E}[\| \nabla f(w; \xi) \|^2] \leq M \| \nabla F(w) \|^2 + N$ with specific $M$, by strong convexity of $F$. We also note that, in \cite{Bach_NIPS2011}, the convergence of SGD without bounded gradient assumption is studied. 
We then use the  new framework for the convergence analysis of SGD to analyze an asynchronous stochastic gradient method.

In~\cite{Hogwild}, an asynchronous stochastic optimization method called Hogwild! was proposed.  Hogwild! algorithm is a parallel version of SGD, where each processor applies SGD steps independently of the other processors to the solution $w$ which is  shared by all processors. Thus, each processor computes a stochastic gradient and updates $w$ without "locking" the memory containing $w$, meaning that multiple processors are able to update $w$ at the same time.  This approach leads to much better scaling of parallel SGD algorithm than a synchoronous version, but the analysis of this method is more complex.   In ~\cite{Hogwild,ManiaPanPapailiopoulosEtAl2015,DeSaZhangOlukotunEtAl2015} various variants of   Hogwild! with a fixed step size are analyzed under the assumption that the gradients are bounded as in  (\ref{bounded_grad_ass}). In this paper, we  extend our analysis of SGD to provide analysis  of Hogwild! with diminishing step sizes and without the   assumption on bounded gradients. 

In a recent technical report \cite{Leblond2018}  Hogwild! with fixed step size is analyzed without the bounded gradient assumption. We note that SGD with fixed step size only converges  to a neighborhood of the optimal solution, while by analyzing the diminishing step size variant we are able to show convergence to the \textit{optimal solution} with probability one. Both in \cite{Leblond2018}  and in this paper, the version of Hogwild! with inconsistent reads and writes is considered.

It is well-known that SGD will converge if a sequence of learning rates $\{\eta_t\}$ satisfies the following conditions (1) $\sum_{t=0}^\infty \eta_t \rightarrow \infty$ and (2) $\sum_{t=0}^\infty \eta^2_t < \infty$. As an important contribution of this paper, we show the convergence of SGD for strongly convex objective function without using bounded gradient assumption when $\{\eta_t\}$ is a diminishing sequence and $\sum_{t=0}^\infty \eta_t \rightarrow \infty$.  In~\cite{Bach_NIPS2011}, the authors also proved the convergence of SGD for $\{\eta_t=\Ocal(1/t^q)\}, 0<q\leq 1,$ without using bounded gradient assumption and the second condition. Compared to~\cite{Bach_NIPS2011}, we prove the convergence of SGD for $\{\eta_t=\Ocal(1/t^q)\}$ which is $1/\mu$ times larger and our proposed class of learning rates satisfying the convergence of SGD is larger. Our proposed class of learning rates satisfying the convergence of SGD is larger than the current state-of-the art one.      

We would like to highlight that this paper is originally from \cite{Nguyen2018_sgdhogwild} (Proceedings of the 35th International Conference on Machine Learning, 2018) but it presents a substantial extension by providing many new results for SGD and Hogwild!.

\subsection{Contribution}

We provide a new framework for the analysis of stochastic gradient algorithms in the strongly convex case under the condition of Lipschitz continuity of the individual function realizations, but {\bf without requiring any bounds on the stochastic gradients}.
Within this framework we have the following contributions: 
\begin{itemize}
\item 
We prove the almost sure (w.p.1) convergence of SGD with diminishing step size. Our analysis provides a larger bound on the possible initial  step size when compared to any previous analysis of convergence in expectation for SGD. 

\item 
We introduce a general recurrence for vector updates which has as its special cases (a) the Hogwild! algorithm with diminishing step sizes, where each update involves all non-zero entries of the computed gradient, and (b) a position-based updating algorithm where each update corresponds to only one uniformly selected non-zero entry of the computed gradient.

\item 
We analyze this general recurrence under inconsistent vector reads from and vector writes to shared memory (where individual vector entry reads and writes are atomic in that they cannot be interrupted by writes to the same entry) assuming that there exists a delay $\tau$ such that during the $(t+1)$-th iteration a gradient of a read vector $w$ is computed which includes the aggregate of all the updates up to and including those made during the $(t-\tau)$-th iteration. In other words, $\tau$ controls to what extent past updates influence the shared memory.
\begin{itemize}
\item 
Our upper bound for the expected convergence rate is $O(1/t)$, and its precise expression allows comparison of algorithms (a) and (b) described above.
\item
For SGD we can improve this upper bound by a factor of 2 and also show that its initial step size can be larger.

\item
We show that $\tau$ can be a function of $t$ as large as 
$\sqrt{ (t/ln t)(1-1/ln t) }$ without affecting the asymptotic behavior of the upper bound; we also determine a constant $T_0$ with the property that, for $t\geq T_0$,  higher order terms containing parameter $\tau$ are smaller than the leading $O(1/t)$ term. We give intuition explaining why the expected convergence rate is not more affected by $\tau$. Our experiments confirm our analysis.
\item 
%
We determine a constant $T_1$ with the property that, for $t\geq T_1$,  the higher order term containing parameter $\|w_0-w_*\|^2$ is smaller than the leading $O(1/t)$ term.
\end{itemize}
\item 
%
All the above contributions generalize to the setting where we do not need to assume that the component functions $f(w;\xi)$ are convex in $w$.
\end{itemize}

Compared to \cite{Nguyen2018_sgdhogwild}, we have following new results:
\begin{itemize}
\item 

We prove the almost sure (w.p.1) convergence of Hogwild! with a diminishing sequence of learning rates $\{\eta_t\}$. 

\item 
We prove the convergence of SGD for diminishing sequences of learning rates $\{\eta_t\}$ with condition $\sum_{t=0}^{\infty} \eta_t \rightarrow \infty$. In other words, we extend the current state-of-the-art class of learning rates satisfying the convergence of SGD.

\item We prove the convergence of SGD for our extended class of learning rates in batch model. 
\end{itemize}




\subsection{Organization}

We analyse the convergence rate of SGD  in Section \ref{analysis} and introduce the general recursion and its analysis in Section \ref{general}. Section~\ref{sec:large_stepsize_convergence} studies the convergence of SGD for our extended class of learning rates. Experiments are reported in Section \ref{sec_experiments}.
\section{New Framework for Convergence Analysis of SGD}
\label{analysis}

We introduce SGD algorithm in Algorithm \ref{sgd_algorithm}. 
\begin{algorithm}[h]
   \caption{Stochastic Gradient Descent (SGD) Method}
   \label{sgd_algorithm}
\begin{algorithmic}
   \STATE {\bfseries Initialize:} $w_0$
   \STATE {\bfseries Iterate:}
   \FOR{$t=0,1,2,\dots$}
  \STATE Choose a step size (i.e., learning rate) $\eta_t>0$. 
  \STATE Generate a random variable $\xi_t$.
  \STATE Compute a stochastic gradient $\nabla f(w_{t};\xi_{t}).$
   \STATE Update the new iterate $w_{t+1} = w_{t} - \eta_t \nabla f(w_{t};\xi_{t})$.
   \ENDFOR
\end{algorithmic}
\end{algorithm}
%

The sequence of random variables $\{\xi_t\}_{t \geq 0}$ is assumed to  be i.i.d.\footnote{Independent and identically distributed.}
Let us introduce our key assumption that each realization $\nabla f(w;\xi)$ is an $L$-smooth function. 


\begin{ass}[$L$-smooth]
\label{ass_smooth}
$f(w;\xi)$ is $L$-smooth for every realization of $\xi$, i.e., there exists a constant $L > 0$ such that, $\forall w,w' \in \mathbb{R}^d$, 
\begin{align*}
\| \nabla f(w;\xi) - \nabla f(w';\xi) \| \leq L \| w - w' \|. \tagthis\label{eq:Lsmooth_basic}
\end{align*} 
\end{ass}

Assumption \ref{ass_smooth} implies that $F$ is also $L$-smooth. Then, by a property of $L$-smooth functions  (in \cite{nesterov2004}), we have, $\forall w, w' \in \mathbb{R}^d$, 
\begin{align*}
F(w) &\leq F(w') + \langle \nabla F(w'),(w-w') \rangle + \frac{L}{2}\|w-w'\|^2. \tagthis\label{eq:Lsmooth}
\end{align*}

The following additional convexity assumption can be made, as it holds for many problems arising in machine learning.

\begin{ass}\label{ass_convex}
$f(w;\xi)$ is convex for every realization of $\xi$, i.e., $\forall w,w' \in \mathbb{R}^d$, 
\begin{gather*}
f(w;\xi)  - f(w';\xi) \geq \langle \nabla f(w';\xi),(w - w') \rangle.
\end{gather*}
\end{ass}

We first derive our analysis under Assumptions  \ref{ass_smooth}, and \ref{ass_convex} and then we derive weaker results under only
Assumption  \ref{ass_smooth}. 

\subsection{Convergence With Probability One}

As discussed in the introduction, under Assumptions \ref{ass_smooth} and \ref{ass_convex} we can now derive a bound on $\mathbb{E}\|\nabla f(w; \xi)\|^2$. 

%
\begin{lem}\label{lem_bounded_secondmoment_04}
Let Assumptions \ref{ass_smooth} and \ref{ass_convex} hold. Then, for $\forall w \in \mathbb{R}^d$, 
\begin{gather}
\mathbb{E}[\|\nabla f(w; \xi)\|^2] \leq  4 L [ F(w) - F(w_{*}) ] + N,
\label{afsfawfwaefwea} 
\end{gather}
where $N = 2 \mathbb{E}[ \|\nabla f(w_{*}; \xi)\|^2 ]$; $\xi$ is a random variable, and $w_{*} = \arg \min_w F(w)$. 
\end{lem}




Using Lemma \ref{lem_bounded_secondmoment_04} and Super Martingale Convergence Theorem \cite{BertsekasSurvey} (Lemma \ref{prop_supermartingale} in Appendix~\ref{useful}), we can provide the sufficient condition for almost sure convergence of Algorithm \ref{sgd_algorithm} in the strongly convex case without assuming any bounded gradients. 

%

\begin{thm}[Sufficient conditions for almost sure convergence]\label{thm_general_02_new_02}
Let Assumptions \ref{ass_stronglyconvex}, \ref{ass_smooth} and \ref{ass_convex} hold. Consider Algorithm \ref{sgd_algorithm} with a stepsize sequence such that
\begin{align*}
0 < \eta_t \leq \frac{1}{2 L} \ , \ \sum_{t=0}^{\infty} \eta_t = \infty \ \text{and} \ \sum_{t=0}^{\infty} \eta_t^2 < \infty. 
\end{align*}
Then, the following holds w.p.1 (almost surely)
\begin{align*}
\| w_{t} - w_{*} \|^2 \to 0. 
\end{align*}

\end{thm}

Note that the classical SGD proposed in \cite{RM1951} has learning rate satisfying conditions
\begin{align*}
\sum_{t=0}^{\infty} \eta_t = \infty \ \text{and} \ \sum_{t=0}^{\infty} \eta_t^2 < \infty
\end{align*}
However, the original analysis is performed under the bounded gradient assumption, as in \eqref{bounded_grad_ass}. 
In Theorem \ref{thm_general_02_new_02}, on the other hand, we do not use this assumption, but instead assume Lipschitz smoothness 
and convexity of the function realizations, which does not contradict the strong convexity of $F(w)$. 

The following result establishes a sublinear convergence rate of SGD.   

\begin{thm}\label{thm_res_sublinear_new_02}
Let Assumptions \ref{ass_stronglyconvex}, \ref{ass_smooth} and \ref{ass_convex} hold. Let $E = \frac{2\alpha L}{\mu}$ with $\alpha=2$. Consider Algorithm \ref{sgd_algorithm} with a stepsize sequence such that $\eta_t = \frac{\alpha}{\mu(t+E)} \leq \eta_0=\frac{1}{2L}$. Then, 
$$ \mathbb{E}[\|w_{t} - w_{*}\|^2] \leq \frac{4\alpha^2 N}{\mu^2} 
\frac{1 }{(t-T+E)} $$
for 
$$t\geq T =\frac{4L}{\mu}\max \{ \frac{L\mu}{N} \|w_{0} - w_{*}\|^2, 1\} - \frac{4L}{\mu},$$
where $N = 2 \mathbb{E}[ \|\nabla f(w_{*}; \xi)\|^2 ]$ and $w_{*} = \arg \min_w F(w)$. 
\end{thm}

\subsection{Convergence Analysis without Convexity}

In this section, we provide the analysis of Algorithm \ref{sgd_algorithm} without using Assumption \ref{ass_convex}, that is, $f(w;\xi)$ is not necessarily convex. We  still do not need to impose the bounded stochastic gradient assumption, since we can derive an analogue of Lemma \ref{lem_bounded_secondmoment_04}, albeit with worse constant in the bound. 
\begin{lem}\label{lem_bounded_secondmoment_04_new}
Let Assumptions \ref{ass_stronglyconvex} and \ref{ass_smooth} hold. Then, for $\forall w \in \mathbb{R}^d$, 
\begin{gather}
\mathbb{E}[\|\nabla f(w; \xi)\|^2] \leq  4L \kappa [ F(w) - F(w_{*}) ] + N,
\label{afsfawfwaefwea_new} 
\end{gather}
where $\kappa = \frac{L}{\mu}$ and $N = 2 \mathbb{E}[ \|\nabla f(w_{*}; \xi)\|^2 ]$; $\xi$ is a random variable, and $w_{*} = \arg \min_w F(w)$.
\end{lem}


Based on the proofs of Theorems \ref{thm_general_02_new_02} and \ref{thm_res_sublinear_new_02}, we can easily have the following two results (Theorems \ref{thm_general_02_new_03} and  \ref{thm_res_sublinear_new_03}). 

\begin{thm}[Sufficient conditions for almost sure convergence]\label{thm_general_02_new_03}
 Let Assumptions \ref{ass_stronglyconvex} and \ref{ass_smooth} hold. Then, we can conclude the statement of Theorem \ref{thm_general_02_new_02} with the definition of the step size replaced by $0 < \eta_t \leq \frac{1}{2L \kappa}$ with $\kappa = \frac{L}{\mu}$.
%
\end{thm}

\begin{thm}\label{thm_res_sublinear_new_03}
Let Assumptions \ref{ass_stronglyconvex} and \ref{ass_smooth}  hold. 
Then, we can conclude the statement of Theorem \ref{thm_res_sublinear_new_02}  with the definition of the step size replaced by $\eta_t = \frac{\alpha}{\mu(t+E)} \leq \eta_0=\frac{1}{2L\kappa}$
 with $\kappa = \frac{L}{\mu}$ and $\alpha=2$, and all other occurrences of $L$ in $E$ and $T$ replaced by $L\kappa$.
%
\end{thm}



\begin{rem}
By strong convexity of $F$, Lemma \ref{lem_bounded_secondmoment_04_new} implies $\mathbb{E}[\|\nabla f(w; \xi)\|^2] \leq  2 \kappa^2 \| \nabla F(w) \|^2 + N$, for $\forall w \in \mathbb{R}^d$, where $\kappa = \frac{L}{\mu}$ and $N = 2 \mathbb{E}[ \|\nabla f(w_{*}; \xi)\|^2 ]$. We can now
substitute  the value $M = 2 \kappa^2$ into Theorem 4.7 in \cite{bottou2016optimization}. We  observe that the resulting initial learning rate in \cite{bottou2016optimization} has to satisfy $\eta_0 \leq \frac{1}{2 L_F \kappa^2}$ while our results allows $\eta_0 = \frac{1}{2 L \kappa}$. We notice that \cite{bottou2016optimization} only assumes that $F$ has Lipschitz continuous gradients with Lipschitz constant $L_F$ while we need the smoothness assumption on individual realizations. Therefore, $L_F$ and $L$ may be different. Both $L_F$ and $L$ values are hard to compare and $L_F$ in \cite{bottou2016optimization} can potentially be much smaller, however, no general comparative statements can be made. 

By introducing Assumption  \ref{ass_smooth}, which holds for many ML problems, we are able to provide the values of $M$ and $N$. Recall that under Assumption \ref{ass_convex}, our initial learning rate is $\eta_0 = \frac{1}{2 L}$ (in Theorem \ref{thm_res_sublinear_new_02}). Thus  Assumption \ref{ass_convex} provides an improvement of the conditions on the learning rate. 
\end{rem}

\section{Asynchronous Stochastic Optimization aka Hogwild!}
\label{general}

Hogwild! \cite{Hogwild} is an asynchronous stochastic optimization method where writes to and reads from vector  positions in shared memory can be inconsistent (this corresponds to (\ref{eqwM2}) as we shall see). 
However, as mentioned in~\cite{ManiaPanPapailiopoulosEtAl2015}, for the purpose of analysis  the method in \cite{Hogwild} performs single vector entry updates that are randomly selected from the non-zero entries of the computed gradient as in (\ref{eqwM1}) (explained later) and requires the assumption of consistent vector reads together with the bounded gradient assumption to prove convergence. Both \cite{ManiaPanPapailiopoulosEtAl2015} and \cite{DeSaZhangOlukotunEtAl2015} prove the same result for fixed step size based on the assumption of bounded stochastic gradients in the strongly convex case but now without assuming consistent vector reads and writes.
In these works the fixed step size $\eta$ must depend on $\sigma$ from the bounded gradient assumption, however, one does not usually know $\sigma$ and thus, we cannot compute a suitable $\eta$ a-priori. 

As claimed by the authors in \cite{ManiaPanPapailiopoulosEtAl2015}, they can eliminate the bounded gradient assumption in their analysis of Hogwild!, which however was only mentioned as a remark without proof. On the other hand, the authors of recent unpublished work~\cite{Leblond2018} formulate and prove, without  the bounded gradient assumption, a precise theorem about the convergence rate of Hogwild! of the form
$$ \mathbb{E}[\|w_{t} - w_* \|^2] \leq (1-\rho)^t(2 \|w_0-w_*\|^2) + b,$$
where $\rho$ is a function of several parameters but independent of the fixed chosen step size $\eta$  and where $b$ is a function of several parameters and has a linear dependency with respect to the fixed step size, i.e., $b=O(\eta)$.

In this section, we discuss  the convergence of Hogwild!  with \textbf{diminishing} stepsize where writes to and reads from vector  positions in shared memory can be  \textbf{inconsistent}. This is a slight modification of the original Hogwild! where the stepsize is fixed. 
In our analysis we also  \textbf{do not use the bounded gradient assumption} as in \cite{Leblond2018}. Moreover, (a) we focus on solving the   \textbf{more general problem} in \eqref{main_prob_expected_risk}, while \cite{Leblond2018} considers the specific case of  the ``finite-sum'' 
	problem in \eqref{main_prob}, and (b) we show that our analysis generalizes to the  \textbf{non-convex case} of the component functions, i.e., we do not need to assume functions $f(w;\xi)$ are convex (we only require $F(w) =  \mathbb{E}[f(w;\xi)]$ to be strongly convex) as opposed to the assumption in \cite{Leblond2018}.

\subsection{Recursion}

We first formulate a general recursion for $w_t$ to which our analysis applies, next we will explain how the different variables in the recursion interact and describe two special cases, and finally we present pseudo code of the algorithm using the recursion.  

The recursion explains which positions in $w_t$ should be updated in order to compute $w_{t+1}$. Since $w_t$ is stored in shared memory and is being updated in a possibly non-consistent way by multiple cores who each perform recursions, the shared memory will contain a vector $w$ whose entries represent a mix of updates. That is, before performing the computation of a recursion, a core will first read  $w$ from shared memory, however, while reading $w$ from shared memory, the entries in $w$ are being updated out of order. The final vector $\hat{w}_t$ read by the core represents an aggregate of a mix of updates in previous iterations.

The general recursion is defined as follows: For $t\geq 0$,
\begin{equation}
 w_{t+1} = w_t - \eta_t d_{\xi_t}  S^{\xi_t}_{u_t} \nabla f(\hat{w}_t;\xi_t),\label{eqwM}
 \end{equation}
 where
 \begin{itemize}
 \item $\hat{w}_t$ represents the vector used in computing the gradient $\nabla f(\hat{w}_t;\xi_t)$ and whose entries have been read (one by one)  from  an aggregate of a mix of  previous updates that led to $w_{j}$, $j\leq t$, and
 \item the $S^{\xi_t}_{u_t}$ are diagonal 0/1-matrices with the property that there exist real numbers $d_\xi$ satisfying
\begin{equation} d_\xi \mathbb{E}[S^\xi_u | \xi] = D_\xi, \label{eq:SexpM} \end{equation}
where the expectation is taken over $u$ and $D_\xi$ is the diagonal 0/1 matrix whose $1$-entries correspond to the non-zero positions in $\nabla f(w;\xi)$ in the following sense: The $i$-th entry of $D_\xi$'s diagonal is equal to 1 if and only if there exists a $w$ such that the $i$-th position of $\nabla f(w;\xi)$ is non-zero. 
\end{itemize}

The role of matrix $S^{\xi_t}_{u_t}$ is that it filters which positions of gradient $\nabla f(\hat{w}_t;\xi_t)$ play a role in (\ref{eqwM}) and need to be computed. Notice that $D_\xi$ represents the support of $\nabla f(w;\xi)$; by $|D_\xi|$ we denote the number of 1s in $D_\xi$, i.e., $|D_\xi|$ equals the size of the support of $\nabla f(w;\xi)$.

We will restrict ourselves to choosing (i.e., fixing a-priori) {\em non-empty} matrices  $S^\xi_u$ that ``partition'' $D_\xi$ in $D$ approximately ``equally sized'' $S^\xi_u$: 
$$ \sum_u S^\xi_u = D_\xi, $$
where each matrix $S^\xi_u$ has either $\lfloor |D_\xi|/D \rfloor$ or $\lceil |D_\xi|/D \rceil$ ones on its diagonal. We uniformly choose one of the matrices $S^{\xi_t}_{u_t}$ in (\ref{eqwM}), hence, $d_\xi$ equals the number of matrices $S^\xi_u$, see (\ref{eq:SexpM}).

In other to explain recursion (\ref{eqwM}) we first consider two special cases. For $D=\bar{\Delta}$, where 
$$ \bar{\Delta} = \max_\xi \{ |D_\xi|\}$$
represents the maximum number of non-zero positions in any gradient computation $f(w;\xi)$, we have that for all $\xi$, there are exactly $|D_\xi|$ diagonal matrices $S^\xi_u$ with a single 1 representing each of the elements in $D_\xi$. Since  $p_\xi(u)= 1/|D_\xi|$ is the uniform distribution, we have $\mathbb{E}[S^\xi_u | \xi] = D_\xi / |D_\xi|$, hence, $d_\xi = |D_\xi|$. This gives the recursion
\begin{equation}
 w_{t+1} = w_t - \eta_t |D_\xi|  [ \nabla f(\hat{w}_t;\xi_t)]_{u_t},\label{eqwM1}
 \end{equation}
 where $ [ \nabla f(\hat{w}_t;\xi_t)]_{u_t}$ denotes the $u_t$-th position of $\nabla f(\hat{w}_t;\xi_t)$ and where $u_t$ is a uniformly selected position that corresponds to a non-zero entry in  $\nabla f(\hat{w}_t;\xi_t)$.
 
At the other extreme, for $D=1$, we have exactly one matrix $S^\xi_1=D_\xi$ for each $\xi$, and we have $d_\xi=1$. This gives the recursion
\begin{equation}
 w_{t+1} = w_t - \eta_t  \nabla f(\hat{w}_t;\xi_t).\label{eqwM2}
 \end{equation}
Recursion (\ref{eqwM2}) represents Hogwild!. In a single-core setting where updates are done in a consistent way and $\hat{w}_t=w_t$ yields SGD.

 
 Algorithm \ref{HogWildAlgorithm} gives the pseudo code corresponding to recursion (\ref{eqwM}) with our choice of sets $S^\xi_u$ (for parameter $D$).
 
 \begin{algorithm}
\caption{Hogwild! general recursion}
\label{HogWildAlgorithm}
\begin{algorithmic}[1]

   \STATE {\bf Input:} $w_{0} \in \R^d$
   \FOR{$t=0,1,2,\dotsc$ {\bf in parallel}} 
    
  \STATE read each position of shared memory $w$
  denoted by $\hat w_t$  {\bf (each position read is atomic)}
  \STATE draw a random sample $\xi_t$ and a random ``filter'' $S^{\xi_t}_{u_t}$
  \FOR{positions $h$ where $S^{\xi_t}_{u_t}$ has a 1 on its diagonal}
   \STATE compute $g_h$ as the gradient $\nabla f(\hat{w}_t;\xi_t)$ at position $h$
   \STATE add $\eta_t d_{\xi_t} g_h$ to the entry at position $h$ of $w$ in shared memory {\bf (each position update is atomic)}
   \ENDFOR
   \ENDFOR
\end{algorithmic}
\end{algorithm}

\subsection{Analysis}

Besides Assumptions \ref{ass_stronglyconvex}, \ref{ass_smooth}, and for now \ref{ass_convex}, we assume the following assumption regarding a parameter $\tau$, called the delay, which indicates which updates in previous iterations have certainly made their way into shared memory $w$.

\begin{ass}[Consistent with delay $\tau$]
\label{ass_tau}
We say that shared memory is consistent with delay $\tau$  with respect to recursion (\ref{eqwM}) if, for all $t$, vector $\hat{w}_t$ includes the aggregate of the updates up to and including those made during the $(t-\tau)$-th iteration (where (\ref{eqwM}) defines the $(t+1)$-st iteration). Each position read from shared memory is atomic and each position update to shared memory is atomic (in that these cannot be interrupted by another update to the same position).
\end{ass}

In other words in the $(t+1)$-th iteration,  $\hat{w}_t$ equals  $w_{t-\tau}$ plus some subset of position updates made during iterations $t-\tau, t-\tau+1, \ldots, t-1$. We assume that there exists a constant delay $\tau$ satisfying Assumption \ref{ass_tau}.

\subsection{Convergence With Probability One}

Appendix \ref{subsec:convergence_Hogwild_wp1} 
proves the following theorem

\begin{thm} [\textbf{Sufficient conditions for almost sure convergence for Hogwild!}]\label{Hogwild:theorem_convergence}
Let Assumptions \ref{ass_stronglyconvex}, \ref{ass_smooth}, \ref{ass_convex} and \ref{ass_tau} hold. Consider Hogwild! method described in Algorithm~\ref{HogWildAlgorithm} with a stepsize sequence such that
\begin{align*}
0 < \eta_t=\frac{1}{LD(2+\beta)(k+t)} < \frac{1}{4LD} , \beta>0, k \geq 3\tau. 
\end{align*}
Then, the following holds w.p.1 (almost surely)
\begin{align*}
\|w_{t} - w_{*} \| \rightarrow 0. 
\end{align*} 
\end{thm}

\subsection{Convergence in Expectation}
Appendix~\ref{subsec_analysis} proves the following theorem where 
$$\bar{\Delta}_D \eqdef D \cdot \mathbb{E}[\lceil |D_\xi|/D \rceil].$$

\begin{thm}
\label{theorem:Hogwild_newnew1}
Suppose Assumptions \ref{ass_stronglyconvex}, \ref{ass_smooth}, \ref{ass_convex} and \ref{ass_tau}  and  consider Algorithm~\ref{HogWildAlgorithm} for sets $S^\xi_u$ with parameter $D$. Let  $\eta_t = \frac{\alpha_t}{\mu(t+E)}$ with $4\leq \alpha_t \leq\alpha$ and $E = \max\{ 2\tau, \frac{4 L \alpha D}{\mu}\}$. Then, the expected number of single vector entry updates after $t$ iterations is equal to
$$t' = t \bar{\Delta}_D /D,$$ and
  \begin{eqnarray*}
 \mathbb{E}[\|\hat{w}_{t} - w_* \|^2] &\leq \frac{4\alpha^2D N}{\mu^2} \frac{t}{(t + E - 1)^2} + O\left(\frac{\ln t}{(t+E-1)^{2}}\right),\\
 \mathbb{E}[\|w_{t} - w_* \|^2]  &\leq \frac{4\alpha^2D N}{\mu^2} \frac{t}{(t + E - 1)^2} + O\left(\frac{\ln t}{(t+E-1)^{2}}\right),
\end{eqnarray*} 
where $N = 2 \mathbb{E}[ \|\nabla f(w_{*}; \xi)\|^2 ]$ and $w_{*} = \arg \min_w F(w)$. 
\end{thm}

In terms of $t'$, the expected number  single vector entry updates after $t$ iterations, $\mathbb{E}[\|\hat{w}_{t} - w_* \|^2]$ and $\mathbb{E}[\|w_{t} - w_* \|^2]$ are at most
$$\frac{4\alpha^2 \bar{\Delta}_D N}{\mu^2} \frac{1}{t'} + O\left(\frac{\ln t'}{t'^{2}}\right).$$

\begin{rem}
In (\ref{eqwM1}) $D=\bar{\Delta}$, hence, $\lceil |D_\xi|/D \rceil =1$ and $\bar{\Delta}_D = \bar{\Delta} = \max_\xi \{|D_\xi|\}$. In (\ref{eqwM2}) $D=1$, hence, $\bar{\Delta}_D = \mathbb{E}[|D_\xi|]$. This shows that the upper bound in Theorem \ref{theorem:Hogwild_newnew1} is better for (\ref{eqwM2}) with $D=1$. If we assume no delay, i.e. $\tau=0$, in addition to $D=1$, then we obtain SGD. Theorem \ref{thm_res_sublinear_new_02} shows that, measured in $t'$, we obtain the upper bound
$$\frac{4\alpha_{SGD}^2 \bar{\Delta}_D N}{\mu^2} \frac{1}{t'} $$
with $\alpha_{SGD}=2$ as opposed to $\alpha\geq 4$.

With respect to parallelism, SGD assumes a single core, while (\ref{eqwM2}) and (\ref{eqwM1}) allow multiple cores. 
Notice that recursion (\ref{eqwM1}) allows us to partition the position of the shared memory among the different processor cores in such a way that each partition can only be updated by its assigned core and where partitions can be read by all cores. This allows optimal resource sharing and could make up for the difference between $\bar{\Delta}_D$ for (\ref{eqwM1}) and (\ref{eqwM2}). We hypothesize that, for a parallel implementation, $D$ equal to a fraction of $\bar{\Delta}$ will lead to best performance.
\end{rem}

\begin{rem}
Surprisingly, the leading term of the upper bound on the convergence rate is independent of delay $\tau$. On one hand, one would expect that a more recent read which contains more of the updates done during the last $\tau$ iterations will lead to better convergence. When inspecting the second order term in the proof in 
Appendix~\ref{subsec_analysis}, we do see that a smaller $\tau$ (and/or smaller sparsity) makes the convergence rate smaller. That is, asymptotically $t$ should be large enough as a function of $\tau$ (and other parameters) in order for the leading term to dominate. 

Nevertheless, in asymptotic terms (for larger $t$) the dependence on $\tau$ is not noticeable.  In fact, 
Appendix~\ref{subsec_sens_tau} shows that we may allow $\tau$ to be a monotonic increasing function of $t$ with
$$\frac{2 L \alpha D}{\mu}\leq \tau(t)\leq \sqrt{t \cdot L(t)},$$
where $L(t)=\frac{1}{\ln t} - \frac{1}{(\ln t)^2}$ (this will make $E = \max\{ 2\tau(t), \frac{4 L \alpha D}{\mu}\}$ also a function of $t$). The leading term of the convergence rate does not change while the second order terms increase to $O(\frac{1}{t\ln t})$. We show that, for
$$ t\geq T_0 =  \exp[ 2\sqrt{\Delta}(1+\frac{(L+\mu)\alpha}{\mu})],$$
where $\Delta= \max_i \Prob \left(   i \in  D_\xi  \right)$ measures sparsity,
the higher order terms that contain $\tau(t)$ (as defined above) are at most the leading term.

Our intuition behind this phenomenon is that for large $\tau$, all the last $\tau$ iterations before the $t$-th iteration use vectors $\hat{w}_j$  with entries that are dominated by the aggregate of updates that happened till iteration $t-\tau$. Since the average sum of the updates during the last $\tau$ iterations is equal to 
 \begin{equation} - \frac{1}{\tau} \sum_{j=t-\tau}^{t-1} \eta_j d_{\xi_j}  S^{\xi_j}_{u_j} \nabla f(\hat{w}_j;\xi_t) \label{Eqtau} \end{equation}
 and all $\hat{w}_j$ look alike in that they mainly represent learned information before the $(t-\tau)$-th iteration, (\ref{Eqtau}) becomes an estimate 
 of the expectation of  (\ref{Eqtau}), i.e.,
 \begin{equation}
 \sum_{j=t-\tau}^{t-1} \frac{-\eta_j}{\tau}  \mathbb{E}[d_{\xi_j}  S^{\xi_j}_{u_j} \nabla f(\hat{w}_j;\xi_t)] 
 =\sum_{j=t-\tau}^{t-1} \frac{-\eta_j}{\tau}  \nabla F(\hat{w}_j). \label{EGD} 
 \end{equation}
 This looks like GD which in the strong convex case has convergence rate $\leq c^{-t}$ for some constant $c>1$. This already shows that larger $\tau$ could help convergence as well. However, 
 estimate (\ref{Eqtau}) has estimation noise with respect to (\ref{EGD}) which explains why in this thought experiment we cannot attain $c^{-t}$ but can only reach a much smaller convergence rate of e.g. $O(1/t)$ as in Theorem \ref{theorem:Hogwild_newnew1}. 
 
 Experiments in Section \ref{sec_experiments} confirm our analysis.
\end{rem}
 
\begin{rem} 
The higher order terms in the proof 
in Appendix~\ref{subsec_analysis} show that, as in Theorem \ref{thm_res_sublinear_new_02},  the expected convergence rate in Theorem \ref{theorem:Hogwild_newnew1} depends on $\|w_0-w_*\|^2$. The proof shows that, for 
$$ t \geq T_1 = \frac{\mu^2}{\alpha^2 N D}\|w_0-w_*\|^2,$$
the higher order term that contains $\|w_0-w_*\|^2$ is at most the leading term. This is comparable to $T$ in Theorem \ref{thm_res_sublinear_new_02} for SGD.
\end{rem}

\begin{rem}
Step size $\eta_t=\frac{\alpha_t}{\mu(t+E)}$ with $4\leq \alpha_t \leq\alpha$ can be chosen to be fixed during periods whose ranges exponentially increase. For $t+E\in [2^h,2^{h+1})$ we define $\alpha_t= \frac{4(t+E)}{2^h}$. Notice that $4\leq \alpha_t<8$ which satisfies the conditions of Theorem \ref{theorem:Hogwild_newnew1} for $\alpha=8$. This means that we can choose 
$$\eta_t = \frac{\alpha_t}{\mu(t+E)}=\frac{4}{\mu 2^h}$$
as step size for $t+E\in [2^h,2^{h+1})$. This choice for $\eta_t$ allows changes in $\eta_t$ to be easily synchronized between cores since these changes only happen when $t+E=2^h$ for some integer $h$. That is, if each core is processing iterations at the same speed, then each core on its own may reliably assume that after having processed $(2^h-E)/P$ iterations the aggregate of all $P$ cores has approximately processed $2^h-E$ iterations. So, after $(2^h-E)/P$ iterations a core will increment its version of $h$ to $h+1$. This will introduce some noise as the different cores will not increment their $h$ versions at exactly the same time, but this only happens during a small interval around every $t+E=2^h$. This will occur rarely for larger $h$.
\end{rem}

\subsection{Convergence Analysis without Convexity}

 In Appendix~\ref{subsec_withoutconvex}, we also show that the proof of Theorem \ref{theorem:Hogwild_newnew1} can easily be modified such that Theorem \ref{theorem:Hogwild_newnew1} with $E\geq \frac{4L\kappa \alpha D}{\mu}$ also holds in the non-convex case of the component functions, i.e., we do not need  Assumption \ref{ass_convex}. Note that this case is not analyzed in \cite{Leblond2018}.

\begin{thm}\label{thm_6}
Let Assumptions \ref{ass_stronglyconvex} and \ref{ass_smooth}  hold. 
Then, we can conclude the statement of Theorem \ref{theorem:Hogwild_newnew1} with $E\geq \frac{4L\kappa \alpha D}{\mu}$ for $\kappa = \frac{L}{\mu}$.
\end{thm}

\begin{thm}[Sufficient conditions for almost sure convergence]\label{thm_general_02_new_044}
 Let Assumptions \ref{ass_stronglyconvex} and \ref{ass_smooth} hold. Then, we can conclude the statement of Theorem \ref{Hogwild:theorem_convergence} with the definition of the step size replaced by $0 < \eta_t=\frac{1}{LD\kappa(2+\beta)(k+t)}$ with $\kappa = \frac{L}{\mu}$.
 \end{thm} 

\section{Convergence of Large Stepsizes}
\label{sec:large_stepsize_convergence}

In \cite{RM1951}, the authors proved the convergence of SGD for step size sequences $\{\eta_t\}$ satisfying conditions
\begin{align*}
\sum_{t=0}^{\infty} \eta_t = \infty \ \text{and} \ \sum_{t=0}^{\infty} \eta_t^2 < \infty.
\end{align*}

In \cite{Bach_NIPS2011}, the authors studied the expected convergence rates for another class of step sizes of $\mathcal{O}(1/t^p)$ where $0< p\leq 1$. This class has many large step sizes in comparison with \cite{RM1951}. For example $\eta_t=\mathcal{O}(1/t^p)$ does not satisfy the second condition (i.e., $\sum_{t=0}^\infty \eta^2_t \rightarrow \infty$) where $0<p<1/2$. In this section, we prove that SGD will converge without using bounded gradient assumption if $\{\eta_t\}$ is a diminishing sequence and $\sum_{t=0}^\infty \eta_t \rightarrow \infty$. Compared to~\cite{Bach_NIPS2011}, we prove the convergence of SGD for step sizes $\eta_t=\Ocal(1/t^q)$ which is $1/\mu$ times larger. Our proposed class is much larger than the classes in \cite{RM1951} and ~\cite{Bach_NIPS2011}.

\subsection{Convergence of Large Stepsizes}
\label{subsec:convergence_largestepsize}
 
The proofs of all theorems and lemmas in this subsection are provided in Appendix~\ref{subsec:convergence_large_stepsize}.

\begin{thm}
\label{theore:SGD_convergence}
Let Assumptions \ref{ass_stronglyconvex}, \ref{ass_smooth}, and \ref{ass_convex} hold. Consider Algorithm \ref{sgd_algorithm} with a step size sequence such that: 
$\eta_t \leq \frac{1}{2L}$, $\eta_t \rightarrow 0$, $\frac{d}{dt}\eta_t \leq 0$ and $\sum_{t=0}^\infty \eta_t \rightarrow \infty$.
Then, $$\mathbb{E}[\| w_{t+1} - w_{*} \|^2 ] \rightarrow 0.$$ 
\end{thm}

Theorem~\ref{theore:SGD_convergence} only discusses about the convergence of SGD for the given step size sequence $\{\eta_t\}$ above. The expected convergence rate of SGD with the setup in Theorem~~\ref{theore:SGD_convergence} is analysed in Theorem~\ref{theorem:convergence_rate}.  

\begin{thm}\label{theorem:convergence_rate}
Let Assumptions \ref{ass_stronglyconvex}, \ref{ass_smooth}, and \ref{ass_convex} hold. Consider Algorithm \ref{sgd_algorithm} with a step size sequence such that $\eta_t \leq \frac{1}{2L}$, $\eta_t \rightarrow 0$, $\frac{d}{dt}\eta_t \leq 0$, and $\sum_{t=0}^\infty \eta_t \rightarrow \infty$.
Then,
\begin{align*}
    \mathbb{E}[\| w_{t+1} - w_{*} \|^2 ] & \leq N\exp(n(0)) 2n(M^{-1}(\ln [ \frac{n(t+1)}{n(0)} ]+ M(t+1))) \\
        & \quad + \exp(-M(t+1)) [\exp(M(1))n^2(0)  N  + \mathbb{E}[\| w_{0} - w_{*} \|^2 ]],
\end{align*}
where $n(t)=\mu \eta_t$ and $M(t)=\int_{x=0}^t n(x)dx$. 
\end{thm}

The upper bound in Theorem~\ref{theorem:convergence_rate} can be interpreted as being approximately equivalent to $\mathbb{E}[\| w_{t} - w_{*} \|^2 ] \leq U n(t-\Delta) + V \exp(-M(t))$ where 
$U=N\exp(n(0))$, $V=\exp(M(1))n^2(0)  N  + \mathbb{E}[\| w_{0} - w_{*} \|^2 ]$ and $\Delta$ is a delay computed from $\ln [ \frac{n(t+1)}{n(0)} ]$. Since $n(t)$ decreases and $M(t)$ increases when $t$ approaches to infinity, $\mathbb{E}[\| w_{t} - w_{*} \|^2 ]]$ decreases in the same way as $n(t)$ to $0$, except for some delay $\Delta$.

As shown in ~\eqref{eq:Yt_Ct_nt} (see also   Appendix~\ref{subsec:convergence_large_stepsize}), we have 
$$
\mathbb{E}[\| w_{t} - w_{*} \|^2 ] \leq A C(t) + B \exp(-M(t)),  
$$ 
where $A$ and $B$ are constants and $C(t)$ is defined in~(\ref{eq:Ct}) below. We show that an alternative proof for the convergence of SGD with the setup above based on the study of $C(t)$ can be developed.

\begin{lem}
\label{lem:Ct}
Let 
\begin{equation}
\label{eq:Ct}
C(t) = \exp(-M(t)) \int_{x=0}^t \exp(M(x)) n(x)^2dx,
\end{equation}
where $\frac{d}{dx}M(x)=n(x)$ with function $n(x)$ satisfying the following conditions:
\begin{enumerate}
\item $\frac{d}{dx}n(x) < 0$,
\item $\frac{d}{dx}n(x)$ is continuous.
\end{enumerate}
Then, there is a moment $T$ such that for all $t>T$, $C(t)>n(t)$.
\end{lem}

\begin{proof}
We take the derivative of $C(t)$, i.e., 
\begin{align*}
\frac{d}{dt}C(t) &= - \exp(-M(t)) n(t) \int_{x=0}^t \exp(M(x)) n(x)^2dx + \exp(-M(t)) \exp(M(t)) n(t)^2\\
&=n(t) [ n(t) - C(t) ] \\
\end{align*}
This shows that
$$ C(t) \mbox{ is decreasing if and only if } C(t)>n(t).$$

Initially $C(0)=0$ and $n(0)>0$, hence, $C(t)$ starts increasing from $t\geq 0$. Since $n(t)$ decreases for all $t\geq 0$, we know that there must exist a first cross over point $x$:
\begin{itemize}
\item There exists a value $x$ such that $C(t)$ increases for $0\leq t<x$, and
\item $C(x)=n(x)$ with derivative $d C(t)/dt |_{t=x} =0$.
\end{itemize}

Since $n(x)$ has a derivative $<0$, we know that $C(t)>n(t)$ immediately after $x$.
Suppose that $C(y)=n(y)$ for some $y>x$ with $C(t)>n(t)$ for $x<t<y$.
This implies that $d C(t)/dt |_{t=y} =0$ and since $d C(t)/dt$ is continuous 
$$ C(y-\epsilon) = C(y) + O(\epsilon^2).$$
Also,
$$ n(y-\epsilon) = n(y) - \epsilon d n(t)/dt |_{t=y} + O(\epsilon^2).$$
Since $d n(t)/dt |_{t=y}<0$, we know that there exists an $\epsilon$ small enough (close to 0) such that
$$ C(y-\epsilon) < n(y-\epsilon).$$
This contradicts $C(t)>n(t)$ for $x<t<y$. We conclude that there does not exist a $y>x$ such that $C(y)=n(y)$:
\begin{itemize}
\item For $t>x$, $C(t)>n(t)$ and $C(t)$ is strictly decreasing.
\end{itemize}

We conclude that for any given $n(t)$, there exists a time $T$ such that $C(t)<n(t)$ for all $t \in [0,T)$, $C(T)=n(T)$ after which $C(t)>n(t)$ when $t\in (T,\infty]$. Note that $C(t)$ is always bigger then zero. 
\end{proof}

As proved above, $C(t)$ decreases for $t>T$. In addition to this we note that $C(t)$ converges to zero when $t$ goes to infinity (see the proof of Theorem~\ref{theore:SGD_convergence} in Appendix~\ref{subsec:convergence_large_stepsize}).  
%
In addition to $C(t)\rightarrow 0$ when $t\rightarrow \infty$, also 
$\exp(-M(t)) \rightarrow 0$ when $t\rightarrow \infty$ because $\sum_{t=0}^\infty n(t) \rightarrow \infty$. Based on these two results we conclude that $\mathbb{E}[\| w_{t} - w_{*} \|^2 ] \rightarrow 0$ when $t\rightarrow \infty$. This is an alternative proof for the convergence of SGD as shown in Theorem~\ref{theore:SGD_convergence}.

\begin{thm}
\label{theorem:fastest_convergence}
Among all stepsizes $\eta_{q,t}=1/(K+t)^q$ where $q>0$, $K$ is a constant such that $\eta_{q,t}\leq \frac{1}{2L}$, SGD algorithm enjoys the fastest convergence with stepsize $\eta_{1,t}=1/(2L+t)$. 
\end{thm} 
 
\subsection{Convergence of Large Stepsizes in Batch Mode}
\label{subsec:convergence_largestepsize_batchmode}

We define 
$$\mathcal{F}_t=\sigma(w_0,\xi'_0,u_0,\dotsc, \xi'_{t-1},u_{t-1}), $$
where
$$\xi'_i = (\xi_{i,1}, \ldots, \xi_{i,k_i}).$$

We consider the following general algorithm with the following gradient updating rule:

\begin{equation}
\label{eq:general_updating}
w_{t+1} = w_t -\eta_t d_{\xi'_{t}} S^{\xi'_t}_{u_t} \nabla f(w_t;\xi'_t), 
\end{equation} 

where $f(w_t;\xi'_t) = \frac{1}{k_t} \sum_{i=1}^{k_t} f(w_t;\xi_{t,i})$.

\begin{thm}
\label{theo:asfasasda}
Let Assumptions \ref{ass_stronglyconvex}, \ref{ass_smooth} and \ref{ass_convex} hold, $\{\eta_t\}$ is a diminishing sequence with conditions $\sum_{t=0}^{\infty} \eta_t \rightarrow \infty$ and $0< \eta_t \leq \frac{1}{2LD}$ for all $t\geq 0$. Then, the sequence $\{w_t\}$ converges to $w_*$ where
$$w_{t+1} = w_t -\eta_t d_{\xi'_{t}} S^{\xi'_t}_{u_t} \nabla f(w_t;\xi'_t).$$ 
\end{thm} 

The proof of Theorem~\ref{theo:asfasasda} is provided in Appendix~\ref{subsec:convergence_large_stepsize_batch_model}.
\section{Numerical Experiments}\label{sec_experiments}

For our numerical experiments, we consider the finite sum minimization problem in \eqref{main_prob}. 
We consider $\ell_2$-regularized logistic regression problems with
\begin{align*}
f_i(w) = \log(1 + \exp(-y_i \langle x_i, w\rangle )) +
\frac{\lambda}{2} \| w \|^2,
\end{align*}
where the penalty parameter $\lambda$ is set to $1/n$, a widely-used
value in  literature \cite{SAG}.
 
We conducted experiments on a single core for Algorithm \ref{HogWildAlgorithm} on two popular datasets \texttt{ijcnn1} ($n = 91, 701$ training data) and \texttt{covtype} ($n = 406,709$ training data) from the LIBSVM\footnote{http://www.csie.ntu.edu.tw/$\sim$cjlin/libsvmtools/datasets/} website. Since we are interested in the expected convergence rate with respect to the number of iterations, respectively number of single position vector updates, we do not need a parallelized multi-core simulation to confirm our analysis. The impact of efficient resource scheduling over multiple cores leads to a  performance improvement complementary to our analysis of (\ref{eqwM}) (which, as discussed, lends itself for an efficient parallelized implementation). We experimented with 10 runs and reported the average results. We choose the step size based on Theorem \ref{theorem:Hogwild_newnew1}, i.e, $\eta_t = \frac{4}{\mu(t+E)}$ and $E = \max\{ 2\tau, \frac{16 L D}{\mu}\}$. For each fraction $v\in \{1,3/4,2/3,1/2,1/3, 1/4\}$ we performed the following experiment: In Algorithm \ref{HogWildAlgorithm}  we  choose each ``filter'' matrix $S^{\xi_t}_{u_t}$ to correspond with a random subset of size $v|D_{\xi_t}|$ of the non-zero positions of $D_{\xi_t}$ (i.e., the support of the gradient corresponding to $\xi_t$). In addition we use $\tau=10$. For the two datasets, 
\begin{figure}[h]
 \centering
 \includegraphics[width=0.45\textwidth]{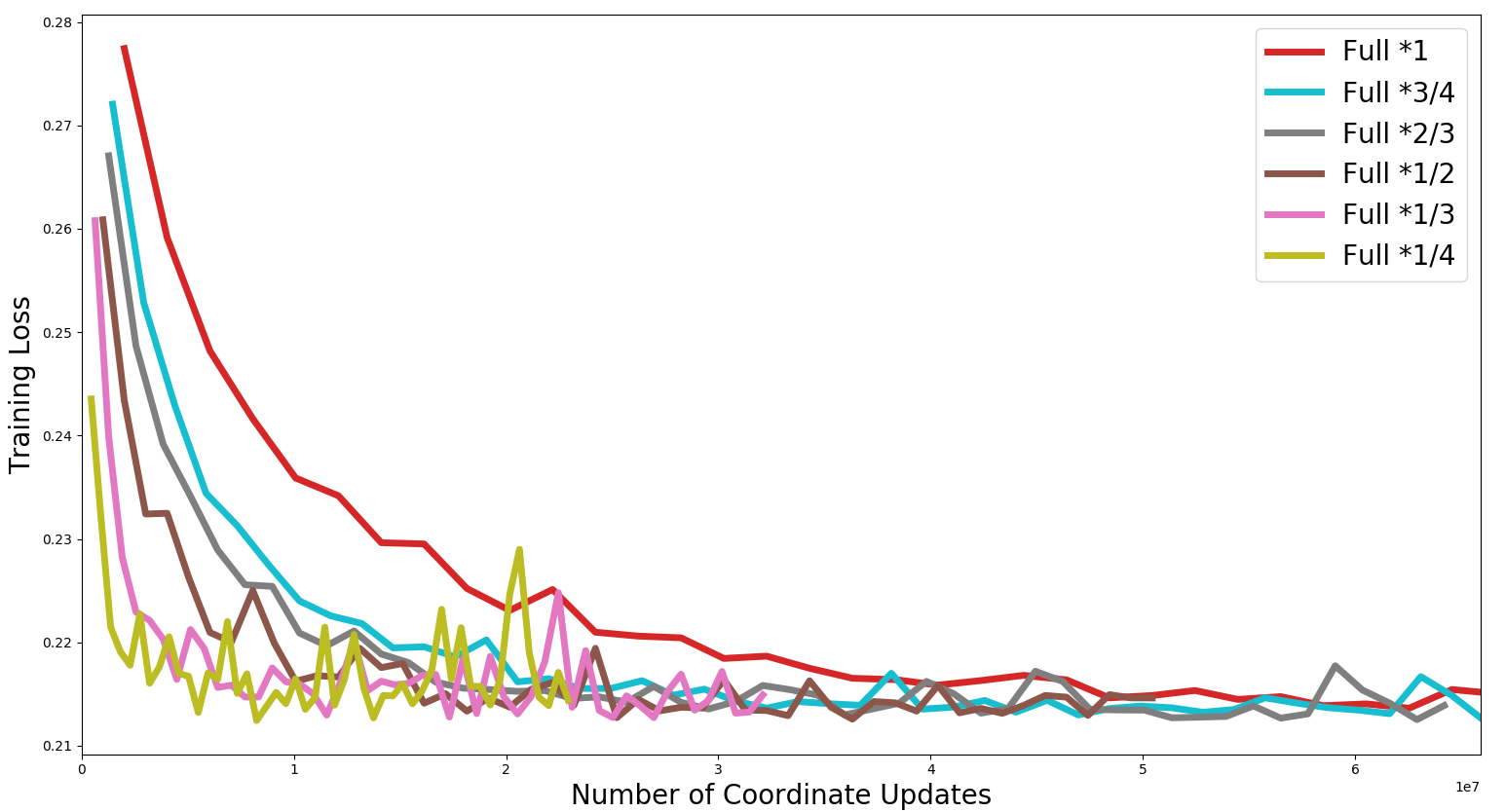}
 \includegraphics[width=0.45\textwidth]{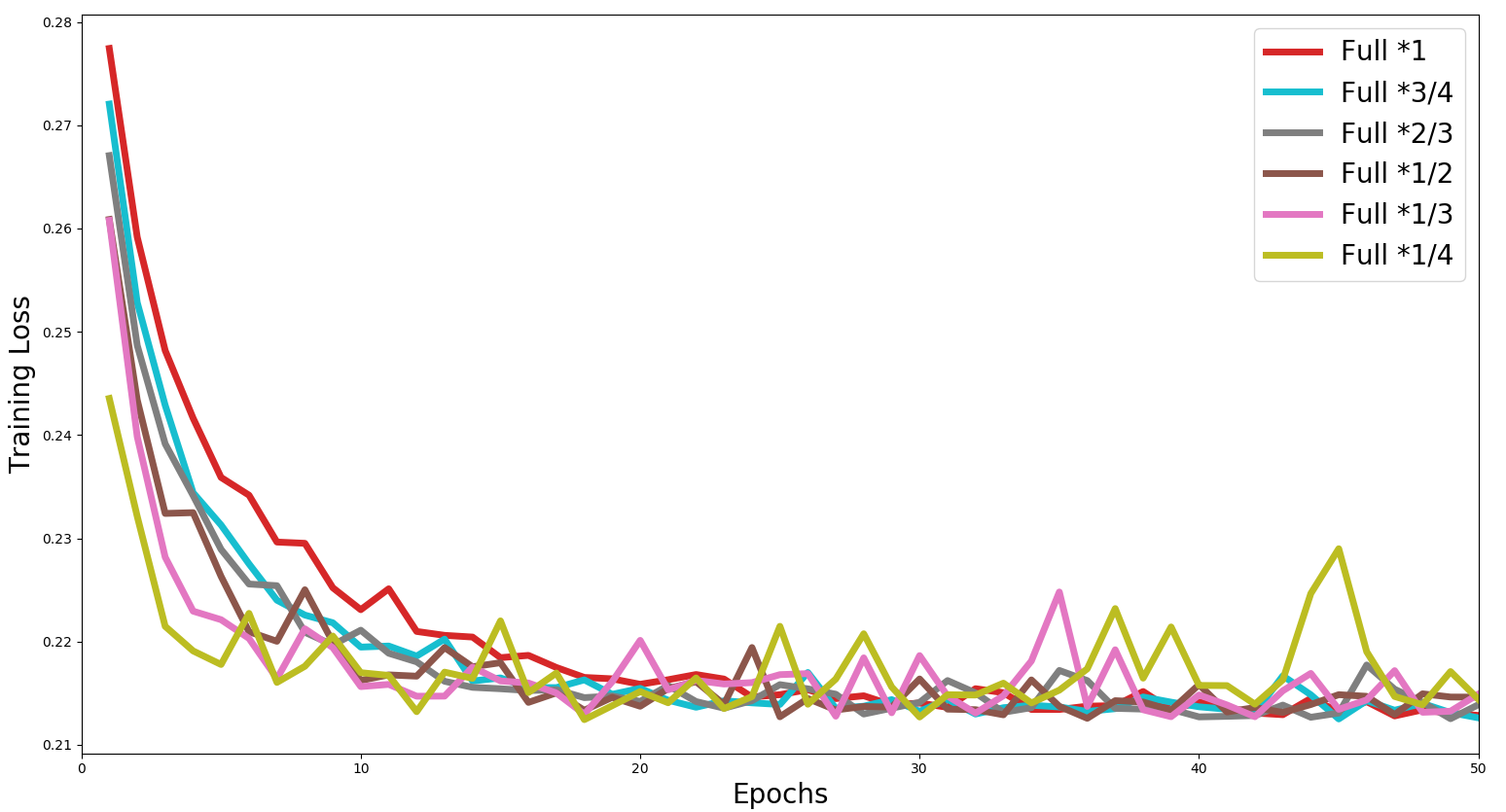}
   \caption{\textit{ijcnn1} for different fraction of non-zero set}
  \label{figure_01_a}
 \end{figure}
 
 \begin{figure}[h]
 \centering
 \includegraphics[width=0.45\textwidth]{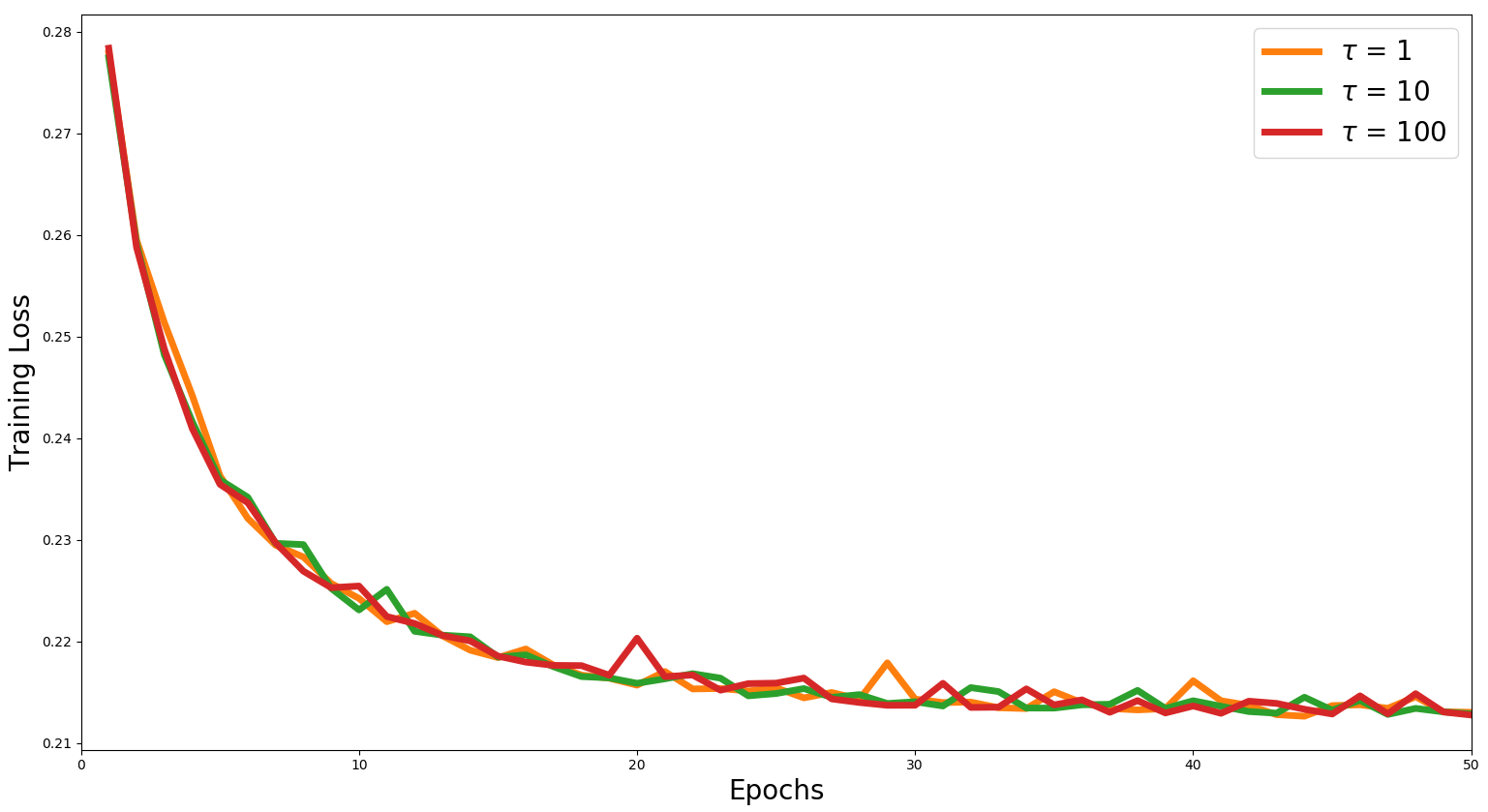} 
   \caption{\textit{ijcnn1} for different $\tau$ with the whole non-zero set}
  \label{figure_01_b}
 \end{figure}

Figures \ref{figure_01_a} and \ref{figure_02_a} plot the training loss for each fraction with $\tau=10$. The top plots have $t'$, the number of coordinate updates, for the horizontal axis. The bottom plots have the number of epochs, each epoch counting $n$ iterations, for the horizontal axis. The results show that each fraction shows a sublinear expected convergence rate of $O(1/t')$; the smaller fractions exhibit larger deviations but do seem to converge faster to the minimum solution.


In Figures \ref{figure_01_b} and \ref{figure_02_b}, we show experiments with different values of $\tau \in \{1, 10, 100\}$ where we use the whole non-zero set of gradient positions (i.e., $v=1$) for the update.  Our analysis states that, for $t= 50$ epochs times $n$ iterations per epoch, $\tau$ can be as large as $\sqrt{t\cdot L(t)}=524$ for \texttt{ijcnn1} and $1058$ for \texttt{covtype}. The experiments indeed show that $\tau\leq 100$ has little effect on the expected convergence rate.

\begin{figure}[h]
 \centering
 \includegraphics[width=0.45\textwidth]{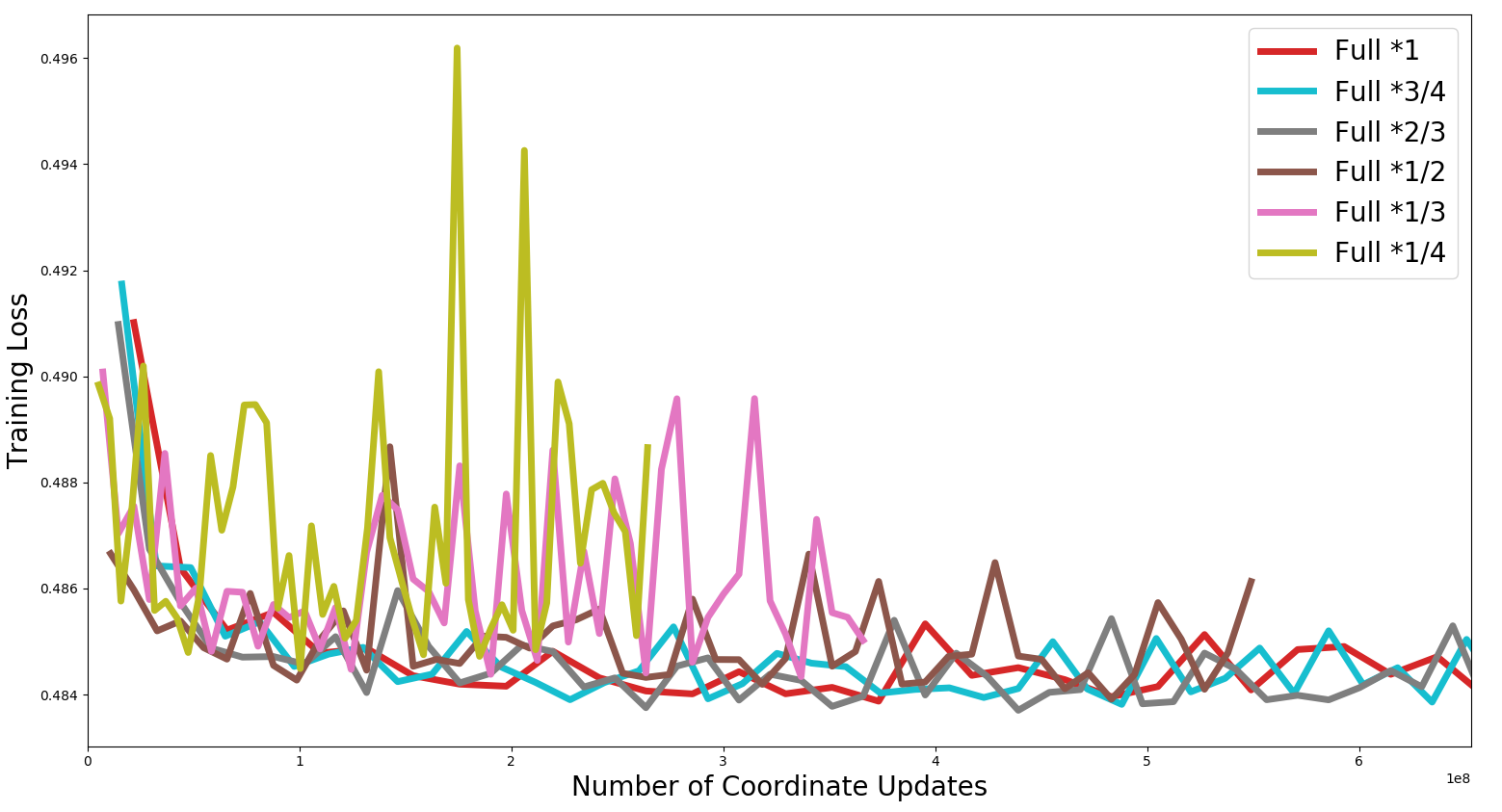}
 \includegraphics[width=0.45\textwidth]{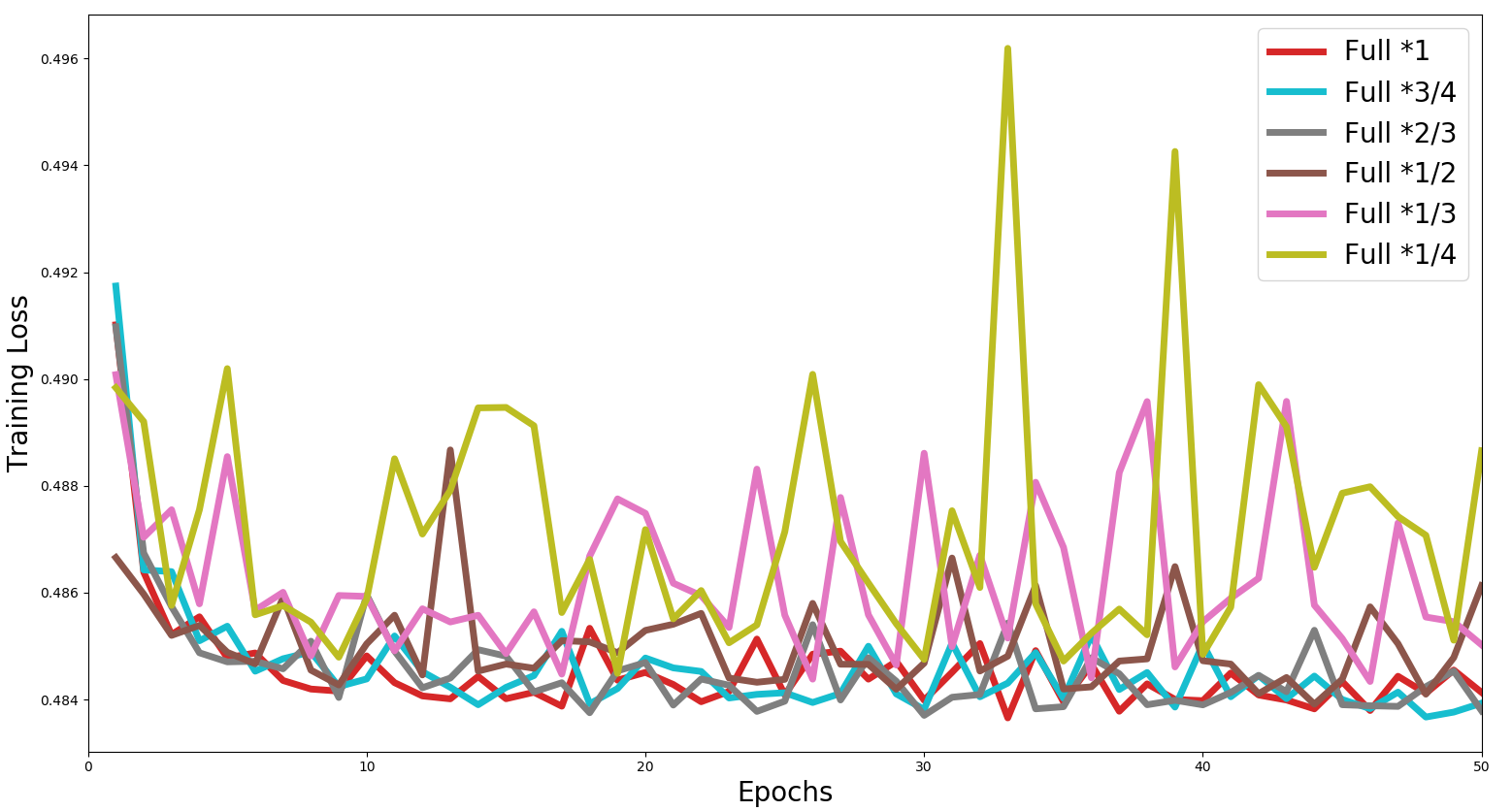}
   \caption{\textit{covtype} for different fraction of non-zero set}
  \label{figure_02_a}
 \end{figure}
 
 \begin{figure}[h]
 \centering
 \includegraphics[width=0.45\textwidth]{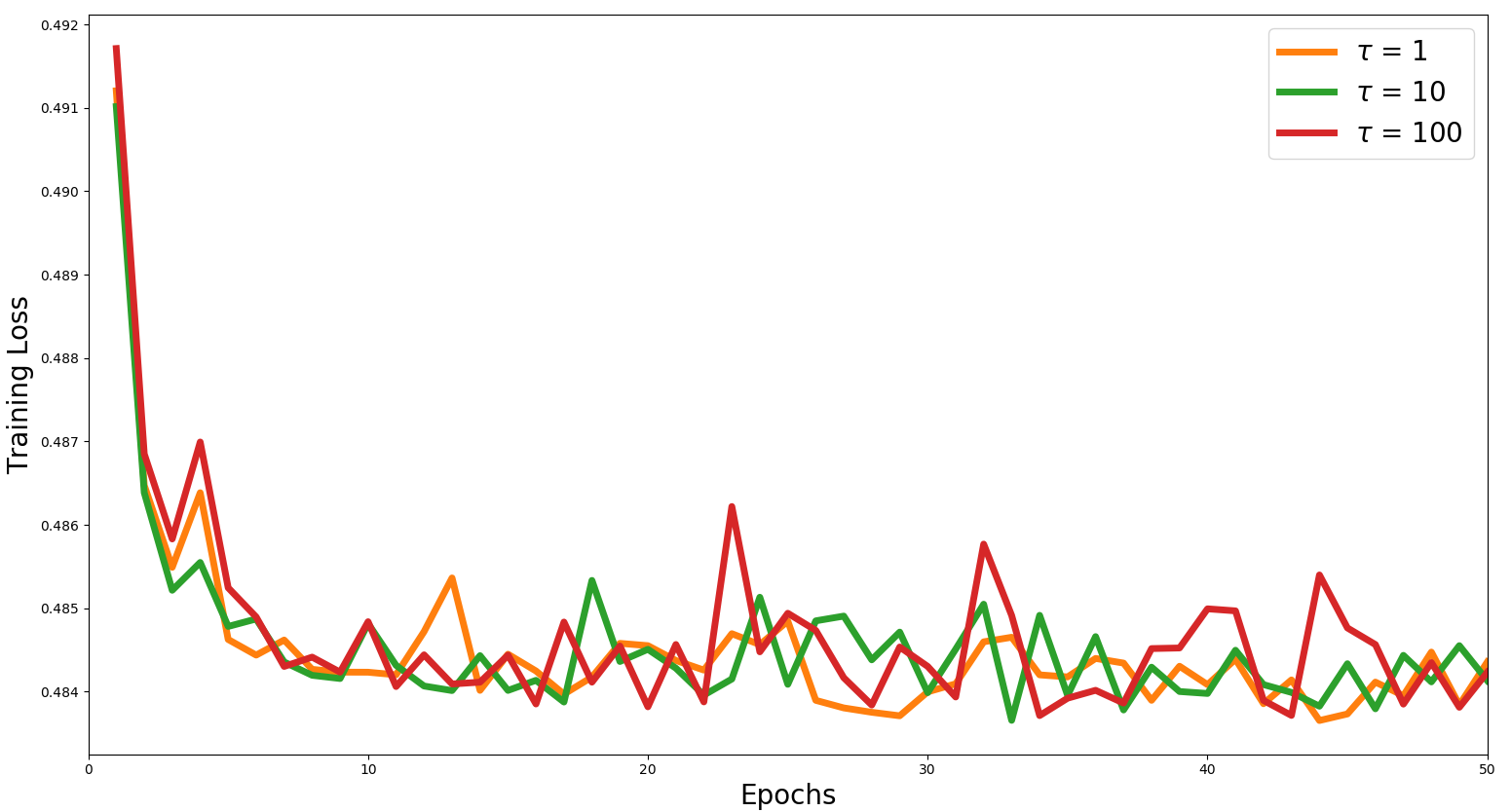} 
   \caption{\textit{covtype} for different $\tau$ with the whole non-zero set}
  \label{figure_02_b}
 \end{figure}

\section{Conclusion}
\label{sec:conclusion}

We have provided the analysis of stochastic gradient algorithms with diminishing step size in the strongly convex case under
the condition of Lipschitz continuity of the individual function realizations, but without requiring any bounds on
the stochastic gradients. We showed almost sure convergence of SGD and provided sublinear upper bounds for the expected convergence rate of a general recursion which includes Hogwild! for inconsistent reads and writes as a special case. We also provided new intuition which will help understanding convergence as observed in practice.

\section*{Acknowledgement}
Phuong Ha Nguyen and Marten van Dijk were supported in part by AFOSR MURI under award number FA9550-14-1-0351. Katya Scheinberg was partially supported by NSF Grants CCF 16-18717 and CCF 17-40796. Martin Tak\'{a}\v{c} was partially supported by the NSF Grant CCF-1618717, CMMI-1663256 and CCF-1740796.

\vskip 0.2in
\bibliography{references}

\clearpage
\appendix

\section{Review of Useful Theorems}\label{useful}

\begin{lem}[Generalization of the result in \cite{SVRG}]\label{lem_bound_diff_grad}
Let Assumptions \ref{ass_smooth} and \ref{ass_convex} hold. Then, $\forall w \in \mathbb{R}^d$, 
\begin{align}\label{eq:001}
\mathbb{E}[ \| \nabla f(w; \xi) - \nabla f(w_{*};\xi) \|^2 ] \leq 2L [ F(w) - F(w_{*})],   
\end{align}
where $\xi$ is a random variable, and $w_{*} = \arg \min_w F(w)$.  
\end{lem}

\begin{lem}[\cite{BertsekasSurvey}]\label{prop_supermartingale}
Let $Y_k$, $Z_k$, and $W_k$, $k = 0,1,\dots$, be three sequences of random variables and let $\{\mathcal{F}_k\}_{k\geq 0}$ be a filtration, that is, $\sigma$-algebras such that $\mathcal{F}_k \subset \mathcal{F}_{k+1}$ for all $k$. Suppose that: 
\begin{itemize}
\item The random variables $Y_k$, $Z_k$, and $W_k$ are nonnegative, and $\mathcal{F}_k$-measurable. 
\item For each $k$, we have $
\mathbb{E}[Y_{k+1} | \mathcal{F}_k] \leq Y_k - Z_k + W_k$. 
\item There holds, w.p.1, 
\begin{gather*}
\sum_{k=0}^{\infty} W_k < \infty. 
\end{gather*}
\end{itemize}
Then, we have, w.p.1,
\begin{gather*}
\sum_{k=0}^{\infty} Z_k < \infty \ \text{and} \ Y_k \to Y \geq 0. 
\end{gather*}
\end{lem}

\section{Proofs of Lemmas \ref{lem_bounded_secondmoment_04} and \ref{lem_bounded_secondmoment_04_new}}

\subsection{Proof of Lemma \ref{lem_bounded_secondmoment_04}}
\textbf{Lemma \ref{lem_bounded_secondmoment_04}}. \textit{Let Assumptions \ref{ass_smooth} and \ref{ass_convex} hold. Then, for $\forall w \in \mathbb{R}^d$,} 
\begin{gather*}
\mathbb{E}[\|\nabla f(w; \xi)\|^2] \leq  4 L [ F(w) - F(w_{*}) ] + N,
\end{gather*}
\textit{where $N = 2 \mathbb{E}[ \|\nabla f(w_{*}; \xi)\|^2 ]$; $\xi$ is a random variable, and $w_{*} = \arg \min_w F(w)$.}

\vspace{.3cm}
\begin{proof}
Note that 
\begin{gather*}
\|a\|^2 = \|a - b + b\|^2 \leq 2\|a - b\|^2 + 2\|b\|^2, \tagthis{\label{eq_aaa001}} \\
\Rightarrow \frac{1}{2}\|a\|^2 - \|b\|^2 \leq \|a - b\|^2. \tagthis{\label{eq_aaa002}}
\end{gather*}


Hence, 
\begin{align*}
\frac{1}{2} \mathbb{E}[ \| \nabla f(w; \xi) \|^2 ] - \mathbb{E}[ \| \nabla f(w_{*}; \xi) \|^2 ] &= \mathbb{E} \left[ \frac{1}{2} \| \nabla f(w; \xi) \|^2 - \| \nabla f(w_{*}; \xi) \|^2 \right] \\ 
& \overset{\eqref{eq_aaa002}}{\leq} \mathbb{E}[\| \nabla f(w; \xi) - \nabla f(w_{*}; \xi) \|^2 ] \\
& \overset{\eqref{eq:001}}{\leq} 2 L [ F(w) - F(w_{*})] \tagthis{\label{eq_aaa005}}
\end{align*}

Therefore,
\begin{align*}
\mathbb{E}[\| \nabla f(w; \xi) \|^2] &\overset{\eqref{eq_aaa001}\eqref{eq_aaa005}}{\leq} 4 L [ F(w) - F(w_{*})] + 2 \mathbb{E}[ \| \nabla f(w_{*}; \xi) \|^2 ]. 
\end{align*}
\end{proof}

\subsection{Proof of Lemma \ref{lem_bounded_secondmoment_04_new}}
\textbf{Lemma \ref{lem_bounded_secondmoment_04_new}}. \textit{Let Assumptions \ref{ass_stronglyconvex} and \ref{ass_smooth} hold. Then, for $\forall w \in \mathbb{R}^d$, 
\begin{gather*}
\mathbb{E}\|\nabla f(w; \xi)\|^2 \leq  4L \kappa [ F(w) - F(w_{*}) ] + N,
\end{gather*}
where $\kappa = \frac{L}{\mu}$ and $N = 2 \mathbb{E}[ \|\nabla f(w_{*}; \xi)\|^2 ]$; $\xi$ is a random variable, and $w_{*} = \arg \min_w F(w)$.}

\vspace{.3cm}
\begin{proof}
Analogous to the proof of Lemma \ref{lem_bounded_secondmoment_04}, we have

Hence, 
\begin{align*}
\frac{1}{2} \mathbb{E}[ \| \nabla f(w; \xi) \|^2 ] - \mathbb{E}[ \| \nabla f(w_{*}; \xi) \|^2 ] &= \mathbb{E} \left[ \frac{1}{2} \| \nabla f(w; \xi) \|^2 - \| \nabla f(w_{*}; \xi) \|^2 \right] \\ 
& \overset{\eqref{eq_aaa002}}{\leq} \mathbb{E}[\| \nabla f(w; \xi) - \nabla f(w_{*}; \xi) \|^2 ] \\
& \overset{\eqref{eq:Lsmooth_basic}}{\leq} L^2 \| w - w_{*} \|^2 \\ 
& \overset{\eqref{eq:stronglyconvex_00}}{\leq} \frac{2L^2}{\mu}[F(w) - F(w_{*})] = 2 L \kappa [F(w) - F(w_{*})]. 
 \tagthis{\label{eq_aaa005_new}}
\end{align*}

Therefore,
\begin{align*}
\mathbb{E}[\| \nabla f(w; \xi) \|^2] \overset{\eqref{eq_aaa001}\eqref{eq_aaa005_new}}{\leq} 4 L \kappa [ F(w) - F(w_{*})] + 2 \mathbb{E}[ \| \nabla f(w_{*}; \xi) \|^2 ]. 
\end{align*}
\end{proof}

\section{Analysis for Algorithm \ref{sgd_algorithm}}\label{sec_analysis_sgd}


In this Section, we provide the analysis of Algorithm \ref{sgd_algorithm} under Assumptions \ref{ass_stronglyconvex}, \ref{ass_smooth}, and \ref{ass_convex}.

We note that if $\{\xi_i\}_{i \geq 0}$ are i.i.d. random variables, then $\mathbb{E}[ \| \nabla f(w_{*}; \xi_0) \|^2 ] = \dots = \mathbb{E}[ \| \nabla f(w_{*}; \xi_t) \|^2 ]$. We have the following results for Algorithm \ref{sgd_algorithm}. 


\vspace{.3cm}

\noindent
\textbf{Theorem \ref{thm_general_02_new_02}} \textbf{(Sufficient condition for almost sure convergence)}. \textit{Let Assumptions \ref{ass_stronglyconvex},   \ref{ass_smooth} and \ref{ass_convex} hold. Consider Algorithm \ref{sgd_algorithm} with a stepsize sequence such that
\begin{align*}
0 < \eta_t \leq \frac{1}{2 L} \ , \ \sum_{t=0}^{\infty} \eta_t = \infty \ \text{and} \ \sum_{t=0}^{\infty} \eta_t^2 < \infty. 
\end{align*}
Then, the following holds w.p.1 (almost surely)
\begin{align*}
\| w_{t} - w_{*} \|^2 \to 0. 
\end{align*}}

\vspace{.3cm}

\begin{proof}
Let $\mathcal{F}_{t} = \sigma(w_{0},\xi_{0},\dots,\xi_{t-1})$ be the $\sigma$-algebra generated by $w_{0},\xi_{0},\dots,\xi_{t-1}$, i.e., $\mathcal{F}_{t}$ contains all the information of $w_{0},\dots,w_{t}$. Note that $\mathbb{E}[\nabla f(w_{t}; \xi_t) | \mathcal{F}_{t}] = \nabla F(w_{t})$. By Lemma \ref{lem_bounded_secondmoment_04}, we have
\begin{gather*}
\mathbb{E}[\|\nabla f(w_{t}; \xi_t)\|^2 | \mathcal{F}_{t} ] \leq  4 L [ F(w_{t}) - F(w_{*}) ] + N, \tagthis \label{ineq:bounded_lemma3_new02}
\end{gather*}
where $N = 2 \mathbb{E}[ \| \nabla f(w_{*}; \xi_0) \|^2 ] = \dots = 2 \mathbb{E}[ \| \nabla f(w_{*}; \xi_t) \|^2 ]$ since $\{\xi_i\}_{i \geq 0}$ are i.i.d. random variables. Note that $w_{t+1} = w_{t} - \eta_t \nabla f(w_{t};\xi_t)$. Hence,
\begin{align*}
\mathbb{E}[\| w_{t+1} - w_{*} \|^2 | \mathcal{F}_{t}] & = \mathbb{E}[ \| w_{t} - \eta_t \nabla f(w_{t};\xi_t) - w_{*} \|^2 | \mathcal{F}_{t}] \\
& = \| w_{t} - w_{*} \|^2 - 2 \eta_t \langle \nabla F(w_{t})  , (w_{t} - w_{*}) \rangle + \eta_t^2 \mathbb{E}[\|\nabla f(w_{t}; \xi_t)\|^2 | \mathcal{F}_{t} ] \\
& \overset{\eqref{eq:stronglyconvex_00}\eqref{ineq:bounded_lemma3_new02}}{\leq} \| w_{t} - w_{*} \|^2 - \mu \eta_t \| w_{t} - w_{*} \|^2 - 2 \eta_t [ F(w_{t}) - F(w_{*}) ] \\&+ 4 L \eta_t^2 [ F(w_{t}) - F(w_{*}) ]  + \eta_t^2 N \\
& = \| w_{t} - w_{*} \|^2 - \mu \eta_t \| w_{t} - w_{*} \|^2 - 2 \eta_t ( 1 - 2 L \eta_t) [ F(w_{t}) - F(w_{*}) ]  + \eta_t^2 N \\
& \leq \| w_{t} - w_{*} \|^2 - \mu \eta_t \| w_{t} - w_{*} \|^2  + \eta_t^2 N.
\end{align*}
The last inequality follows since $0 < \eta_t \leq \frac{1}{2L}$. Therefore,
\begin{align*}
\mathbb{E}[\| w_{t+1} - w_{*} \|^2 | \mathcal{F}_{t}]  \leq \| w_{t} - w_{*} \|^2 - \mu \eta_t \| w_{t} - w_{*} \|^2  + \eta_t^2 N. \tagthis \label{main_ineq_sgd_new02}
\end{align*}
Since $\sum_{t=0}^{\infty} \eta_t^2 N < \infty$, 
we could apply Lemma \ref{prop_supermartingale}. Then, we have w.p.1,
\begin{gather*}
\| w_{t} - w_{*} \|^2 \to W \geq 0, \\ \text{and} \
\sum_{t=0}^{\infty} \mu \eta_t \| w_{t} - w_{*} \|^2 < \infty. 
\end{gather*}

We want to show that $\| w_{t} - w_{*} \|^2 \to 0$, w.p.1. Proving by contradiction, we assume that there exist $\epsilon > 0$ and $t_0$, s.t. $\| w_{t} - w_{*} \|^2 \geq \epsilon$ for $\forall t \geq t_0$. Hence, 
\begin{align*}
\sum_{t=t_0}^{\infty} \mu \eta_t \| w_{t} - w_{*} \|^2 \geq \mu \epsilon \sum_{t=t_0}^{\infty} \eta_t = \infty. 
\end{align*}
This is a contradiction. Therefore, $\|w_{t} - w_{*}\|^2 \to 0$ w.p.1. 
\end{proof}

 
 \noindent
\textbf{Theorem \ref{thm_res_sublinear_new_02}}. \textit{Let Assumptions \ref{ass_stronglyconvex}, \ref{ass_smooth} and \ref{ass_convex} hold. Let $E = \frac{2\alpha L}{\mu}$ with $\alpha=2$. Consider Algorithm \ref{sgd_algorithm} with a stepsize sequence such that $\eta_t = \frac{\alpha}{\mu(t+E)} \leq \eta_0=\frac{1}{2L}$. Then,}
$$ \mathbb{E}[\|w_{t} - w_{*}\|^2] \leq \frac{4\alpha^2 N}{\mu^2} 
\frac{1 }{(t-T+E)} $$
\textit{for} $t\geq T =\frac{4L}{\mu}\max \{ \frac{L\mu}{N} \|w_{0} - w_{*}\|^2, 1\} - \frac{4L}{\mu}$, \textit{where} $N = 2 \mathbb{E}[ \|\nabla f(w_{*}; \xi)\|^2 ]$ \textit{and} $w_{*} = \arg \min_w F(w)$.  

\vspace{.3cm}

\begin{proof}
Using the beginning of the proof of Theorem \ref{thm_general_02_new_02}, taking the expectation to \eqref{main_ineq_sgd_new02}, with $0 < \eta_t \leq \frac{1}{2L}$, we have
\begin{align*}
\mathbb{E}[\| w_{t+1} - w_{*} \|^2 ]  \leq (1 - \mu \eta_t) \mathbb{E}[ \| w_{t} - w_{*} \|^2]  + \eta_t^2 N.   
\end{align*} 

We first show that
\begin{equation} \mathbb{E}[\|w_{t} - w_{*}\|^2]\leq \frac{N}{\mu^2} G
\frac{1 }{(t+E)}, \label{eqGGG} \end{equation}
where 
$G=\max\{I,J\}$, and 
\begin{align*}
I &= \frac{E \mu^2}{N} \mathbb{E}[\|w_{0} - w_{*}\|^2]>0, \\
J &= \frac{\alpha^2}{\alpha - 1}>0.
\end{align*}

We use mathematical induction to prove (\ref{eqGGG}) (this trick is based on the idea from \cite{bottou2016optimization}). Let $t = 0$, we have 
\begin{align*}
\mathbb{E}[\|w_{0} - w_{*}\|^2] 
\leq \frac{NG}{\mu^2 E},  
\end{align*}
which is obviously true since 
$
G \geq \frac{E\mu^2}{N} \|w_{0} - w_{*}\|^2.
$

Suppose it is true for $t$, we need to show that it is also true for $t+1$. We have
\begin{align*}
\mathbb{E}[\|w_{t+1} - w_{*}\|^2]  & \leq \left(1 - \frac{\alpha}{t+E} \right) \frac{NG}{\mu^2(t+E)} + \frac{\alpha^2 N}{\mu^2(t+E)^2} \\
& = \left(\frac{t +E - \alpha}{\mu^2(t+E)^2} \right) NG + \frac{\alpha^2N}{\mu^2(t+E)^2} \\
& = \left(\frac{t +E - 1}{\mu^2(t+E)^2} \right) NG  - \left(\frac{\alpha - 1}{\mu^2(t+E)^2} \right) NG + \frac{\alpha^2 N}{\mu^2(t+E)^2}.
\end{align*}
Since
$
G \geq \frac{\alpha^2}{\alpha - 1},
$
$$ - \left(\frac{\alpha - 1}{\mu^2(t+E)^2} \right) NG + \frac{\alpha^2 N}{\mu^2(t+E)^2}\leq 0.$$
This implies
\begin{align*}
\mathbb{E}[\|w_{t+1} - w_{*}\|^2]  & \leq \left(\frac{t+E - 1}{\mu^2(t+E)^2} \right) NG \\
& = \left(\frac{(t+E)^2 - 1}{(t+E)^2} \right) \frac{NG}{\mu^2(t +E+ 1)} \\
& \leq\frac{NG}{\mu^2(t +E+ 1)}. 
\end{align*}
This proves (\ref{eqGGG}) by induction in $t$.

Notice that the induction proof of (\ref{eqGGG}) holds more generally for $E\geq \frac{2\alpha L}{\mu}$ with $\alpha>1$ (this is sufficient for showing $\eta_t \leq \frac{1}{2L}$. In this more general interpretation we can see that the convergence rate is minimized for $I$ minimal, i.e., $E=\frac{2\alpha L}{\mu}$ and for this reason we have fixed $E$ as such in the theorem statement.

Notice that 
$$G=\max\{I,J\} = \max \{ \frac{2 \alpha L \mu}{N} \mathbb{E}[\|w_{0} - w_{*}\|^2], \frac{\alpha^2}{\alpha - 1} \}.$$
We choose $\alpha=2$ such that $\eta_t$ only depends on known parameters $\mu$ and $L$. For this $\alpha$ we obtain
$$G = 4 \max \{\frac{L \mu}{N} \mathbb{E}[\|w_{0} - w_{*}\|^2], 1\}.$$

For
$T =\frac{4L}{\mu}\max \{ \frac{L\mu}{N} \mathbb{E}[\|w_{0} - w_{*}\|^2], 1\} - \frac{4L}{\mu}$, 
 we have that
according to (\ref{eqGGG})
\begin{eqnarray}
 \frac{L \mu}{N} \mathbb{E}[\|w_{T} - w_{*}\|^2]&\leq& \frac{L \mu}{N}\frac{N}{\mu^2} 
\frac{G }{(T+E)} \nonumber \\
&=& \frac{L}{\mu} \frac{ 4 \max \{\frac{L \mu}{N} \mathbb{E}[\|w_{0} - w_{*}\|^2], 1\}}
{\frac{4L}{\mu}\max \{\frac{L\mu}{N} \mathbb{E}[\|w_{0} - w_{*}\|^2], 1\} } =
1. \label{eqTT} \end{eqnarray}
Applying (\ref{eqGGG})with $w_T$ as starting point rather than $w_0$ gives, for $t\geq \max\{T,0\}$,
$$\mathbb{E}[\|w_{t} - w_{*}\|^2]\leq \frac{N}{\mu^2} G
\frac{1 }{(t-T+E)},$$
where $G$ is now equal to
$$ 4 \max \{ \frac{L \mu}{N} \mathbb{E}[\|w_{T} - w_{*}\|^2], 1\},$$
which equals $4$, see (\ref{eqTT}). For any given $w_0$, we prove the theorem. 
\end{proof}


\section{Analysis for Algorithm~\ref{HogWildAlgorithm}}
\label{sec:Hogwild_insconsistent_read_write}

\subsection{Recurrence and Notation}

We introduce the following notation: For each $\xi$, we define $D_\xi \subseteq \{1,\ldots, d\}$ as the set of possible non-zero positions in a vector of the form $\nabla f(w;\xi)$ for some $w$. We consider a fixed mapping from $u\in U$ to subsets $S^{\xi}_u\subseteq D_\xi$ for each possible $\xi$. In our notation we also let $D_\xi$ represent the diagonal $d\times d$ matrix with ones exactly at the positions corresponding to $D_\xi$ and with zeroes elsewhere. Similarly, $S^{\xi}_u$ also denotes a diagonal matrix with ones at the positions corresponding to $D_\xi$.

We will use a probability distribution $p_\xi(u)$ to indicate how to randomly select a matrix $S^\xi_u$. We choose the matrices $S^\xi_u$ and distribution $p_\xi(u)$ so that there exist $d_\xi$ such that
\begin{equation} d_\xi \mathbb{E}[S^\xi_u | \xi] = D_\xi, \label{eq:Sexp} \end{equation}
where the expectation is over $p_\xi(u)$.

We will restrict ourselves to choosing {\em non-empty} sets  $S^\xi_u$ that partition $D_\xi$ in $D$ approximately equally sized sets together with uniform distributions $p_\xi(u)$ for some fixed $D$. So, if $D\leq |D_\xi|$, then sets have sizes $\lfloor |D_\xi|/D \rfloor$ and $\lceil |D_\xi|/D \rceil$. For the special case $D>|D_\xi|$ we have exactly $|D_\xi|$ singleton sets of size $1$ (in our definition we only use non-empty sets).

For example, for $D=\bar{\Delta}$, where 
$$ \bar{\Delta} = \max_\xi \{ |D_\xi|\}$$
represents the maximum number of non-zero positions in any gradient computation $f(w;\xi)$, we have that for all $\xi$, there are exactly $|D_\xi|$ singleton sets $S^\xi_u$ representing each of the elements in $D_\xi$. Since  $p_\xi(u)= 1/|D_\xi|$ is the uniform distribution, we have $\mathbb{E}[S^\xi_u | \xi] = D_\xi / |D_\xi|$, hence, $d_\xi = |D_\xi|$.
As another example at the other extreme, for $D=1$, we have exactly one set $S^\xi_1=D_\xi$ for each $\xi$. Now $p_\xi(1)=1$ and we have $d_\xi=1$.

We define the parameter
$$ \bar{\Delta}_D \eqdef D \cdot \mathbb{E}[\lceil |D_\xi|/D \rceil],$$
where the expectation is over $\xi$.
We use $\bar{\Delta}_D$ in the leading asymptotic term for the convergence rate in our main theorem. We observe that
$$ \bar{\Delta}_D \leq \mathbb{E}[|D_\xi|] + D-1$$
and
$\bar{\Delta}_D \leq \bar{\Delta}$ with equality for $D=\bar{\Delta}$.

For completeness we define
\begin{equation}
\Delta \eqdef 
\max_{i}
\Prob \left(   i \in  D_\xi  \right). \nonumber
\end{equation}
Let us remark, that $\Delta \in (0,1]$ measures the probability of collision.
Small $\Delta$ means that 
there is a small chance that the support of two random realizations of $\nabla f(w;\xi)$
will have an intersection.
On the other hand, $\Delta = 1$ means that almost surely, the support of two stochastic gradients will have non-empty intersection.

With this definition of $\Delta$ it is an easy exercise to show that 
for iid $\xi_1$ and $\xi_2$ in a finite-sum setting (i.e., $\xi_i$ and $\xi_2$ can only take on a finite set of possible values) we have
\begin{align*}
&\Exp[|\ve{ \nabla f(w_1; \xi_1) }{\nabla f(w_2; \xi_2)}|] 
\\& \leq 
\frac{\sqrt \Delta}2
\left(
\Exp[\|  \nabla f(w_1; \xi_1) \|^2] + \Exp[\|\nabla f(w_2; \xi_2)\|^2]
\right)
\tagthis \label{Masdfasdfasfa}
\end{align*}
(see Proposition 10 in \cite{Leblond2018}).
We notice that in the non-finite sum setting we can use the property that for any two vectors $a$ and $b$, $\langle a, b\rangle \leq (\|a\|^2 + \|b\|^2)/2$ and this proves (\ref{Masdfasdfasfa}) with $\Delta$ set to $\Delta=1$. In our asymptotic analysis of the convergence rate, we will show how $\Delta$ plays a role in non-leading terms -- this, with respect to the leading term, it will not matter whether we use $\Delta=1$ or $\Delta$ equal the probability of collision (in the finite sum case).






We have
\begin{equation}
 w_{t+1} = w_t - \eta_t d_{\xi_t}  S^{\xi_t}_{u_t} \nabla f(\hat{w}_t;\xi_t),\label{eqw}
 \end{equation}
where  $\hat{w}_t$ represents the vector used in computing the gradient $\nabla f(\hat{w}_t;\xi_t)$ and whose entries have been read (one by one)  from  an aggregate of a mix of  previous updates that led to $w_{j}$, $j\leq t$.
Here, we assume that
\begin{itemize}
\item updating/writing to vector positions is atomic, reading vector positions is atomic, and
\item there exists a ``delay'' $\tau$ such that, for all $t$, vector $\hat{w}_t$ includes all the updates up to and including those made during the $(t-\tau)$-th iteration (where (\ref{eqw}) defines the $(t+1)$-st iteration).
\end{itemize}
Notice that we do {\bf not assume consistent reads and writes of vector positions}. We only assume that up to a ``delay'' $\tau$ all writes/updates are included in the values of positions that are being read.

According to our definition of $\tau$, in (\ref{eqw}) vector $\hat{w}_t$ represents an inconsistent read with entries that contain all of the updates made during the $1$st to $(t-\tau)$-th iteration. Furthermore each entry in $\hat{w}_t$ includes some of the updates made during the $(t-\tau+1)$-th iteration up to $t$-th iteration. Each entry includes its own subset of updates because writes are inconsistent.  We model this by ``masks'' $\Sigma_{t,j}$ for $t-\tau\leq j\leq t-1$. A mask $\Sigma_{t,j}$ is a diagonal 0/1-matrix with the 1s expressing which of the entry updates made in the $(j+1)$-th iteration are included in $\hat{w}_t$.
That is,
 \begin{equation}
\label{eq:what_wrhot_1}
\hat{w}_t = w_{t-\tau} - \sum_{j=t-\tau}^{t-1} \eta_j d_{\xi_j} \Sigma_{t,j}  S^{\xi_j}_{u_j} \nabla  f(\hat{w}_j;\xi_j)  .
\end{equation}


Notice that the recursion (\ref{eqw}) implies
\begin{equation}
\label{eq:wt_wrhot_1}
w_t = w_{t-\tau} - \sum_{j=t-\tau}^{t-1} \eta_j d_{\xi_j}  S^{\xi_j}_{u_j} \nabla  f(\hat{w}_j;\xi_j) .
\end{equation}
By combining~\eqref{eq:wt_wrhot_1} and~\eqref{eq:what_wrhot_1} we obtain
\begin{equation}
\label{eq:wt_what_2}
w_t-\hat{w}_t = - \sum_{j=t-\tau}^{t-1}  \eta_j d_{\xi_j} (I-\Sigma_{t,j})  S^{\xi_j}_{u_j}  \nabla f(\hat{w}_j;\xi_j) ,
\end{equation}
where $I$ represents the identity matrix.


\subsection{Main Analysis}\label{subsec_analysis}

We first derive a couple lemmas which will help us deriving our main bounds. In what follows let Assumptions \ref{ass_stronglyconvex}, \ref{ass_smooth}, \ref{ass_convex} and  \ref{ass_tau} hold for all lemmas. We define 
$$\mathcal{F}_t=\sigma(w_0,\xi_0,u_0,\sigma_0,\dotsc, \xi_{t-1},u_{t-1},\sigma_{t-1}), $$
where
$$\sigma_{t-1} = (\Sigma_{t,t-\tau}, \ldots, \Sigma_{t,t-1}).$$
When we subtract $\tau$ from, for example, $t$ and write $t-\tau$, we will actually mean $\max\{t-\tau, 0\}$.

 \begin{lem} \label{lem:expect}  
 We have
 $$\mathbb{E}[\| d_{\xi_t}S^{\xi_t}_{u_t} \nabla  f(\hat{w}_t;\xi_t)  \|^2| \mathcal{F}_t, \xi_t]\leq  D \|\nabla f(\hat{w}_t;\xi_t) \|^2$$
 and
 $$\mathbb{E}[d_{\xi_t}S^{\xi_t}_{u_t} \nabla  f(\hat{w}_t;\xi_t) | \mathcal{F}_t ] 
= \nabla F(\hat{w}_t).$$
 \end{lem}
 
 \vspace{.3cm}
 
 \begin{proof}
 For the first bound, if we take the expectation of $\| d_{\xi_t}S^{\xi_t}_{u_t} \nabla  f(\hat{w}_t;\xi_t)  \|^2$ with respect to $u_t$, then we have (for vectors $x$ we denote the value if its $i$-th position by $[x]_i$)
 \begin{align*}
& \mathbb{E}[\|d_{\xi_t}S^{\xi_t}_{u_t} \nabla  f(\hat{w}_t;\xi_t)  \|^2| \mathcal{F}_t, \xi_t] = 
 d_{\xi_t}^2 \sum_u p_{\xi_t}(u)  \|S^{\xi_t}_{u} \nabla  f(\hat{w}_t;\xi_t)  \|^2 = \\
&= 
 d_{\xi_t}^2 \sum_u p_{\xi_t}(u)  \sum_{i\in S^{\xi_t}_{u}} [\nabla f(\hat{w}_t;\xi_t)]_i^2 
 d_{\xi_t} \sum_{i\in D_{\xi_t}} [\nabla f(\hat{w}_t;\xi_t)]_i^2 =
 d_{\xi_t} \| f(\hat{w}_t;\xi_t) \|^2 \leq  D \|\nabla f(\hat{w}_t;\xi_t) \|^2,
\end{align*}
where the transition to the second line follows from (\ref{eq:Sexp}).

For the second bound, if we take the expectation of $ d_{\xi_t}S^{\xi_t}_{u_t} \nabla  f(\hat{w}_t;\xi_t)$ wrt $u_t$, then we have:
\begin{align*}
\mathbb{E}[ d_{\xi_t}S^{\xi_t}_{u_t} \nabla  f(\hat{w}_t;\xi_t) | \mathcal{F}_t, \xi_t] &=
 d_{\xi_t} \sum_u p_{\xi_t}(u) S^{\xi_t}_{u}  \nabla f(\hat{w}_t;\xi_t) = D_{\xi_t} \nabla  f(\hat{w}_t;\xi_t)= 
    \nabla f(\hat{w}_t;\xi_t),
\end{align*}
and this can be used to derive
\begin{align*}
\mathbb{E}[ d_{\xi_t}S^{\xi_t}_{u_t}  f(\hat{w}_t;\xi_t) | \mathcal{F}_t]
=
\mathbb{E}[\mathbb{E}[d_{\xi_t}S^{\xi_t}_{u_t}  f(\hat{w}_t;\xi_t) | \mathcal{F}_t, \xi_t] | \mathcal{F}_t] &=  \nabla F(\hat{w}_t).  
\end{align*}
\end{proof}

As a consequence of this lemma we derive a bound on the expectation of  $\|w_t- \hat{w}_t \|^2$.

\begin{lem}
\label{lemma:hogwild_21} The expectation of $\|w_t- \hat{w}_t \|^2$ is at most
$$\mathbb{E}[\|w_t- \hat{w}_t \|^2 ] \leq  (1+\sqrt{\Delta}\tau) D \sum_{j=t-\tau}^{t-1} \eta_j^2 (2L^2 \mathbb{E}[\| \hat{w}_j -w_{*} \|^2] +N). $$
\end{lem}

\begin{proof}
As shown in~\eqref{eq:wt_what_2}, 
$$
w_t-\hat{w}_t = - \sum_{j=t-\tau}^{t-1}  \eta_j d_{\xi_j} (I-\Sigma_{t,j})  S^{\xi_j}_{u_j} \nabla  f(\hat{w}_j;\xi_j) .
$$
This can be used to derive an expression for the square of its norm: $\|w_t- \hat{w}_t \|^2$
\begin{eqnarray*}
&=& \|  \sum_{j=t-\tau}^{t-1}  \eta_j d_{\xi_j} (I-\Sigma_{t,j})  S^{\xi_j}_{u_j}  \nabla  f(\hat{w}_j;\xi_j) \|^2 \\
&=& \sum_{j=t-\tau}^{t-1} \| \eta_j d_{\xi_j}    (I-\Sigma_{t,j}) S^{\xi_j}_{u_j} \nabla   f(\hat{w}_j;\xi_j) \|^2   \\
&& + \sum_{i\neq j \in \{t-\tau, \ldots, t-1\}}
\ve{ \eta_j d_{\xi_j}   (I-\Sigma_{t,j}) S^{\xi_j}_{u_j} \nabla   f(\hat{w}_j;\xi_j)}{ \eta_i d_{\xi_i}   (I-\Sigma_{t,j}) S^{\xi_i}_{u_i}  \nabla  f(\hat{w}_i;\xi_i) }.
\end{eqnarray*}

Applying (\ref{Masdfasdfasfa}) to the inner products implies
%
%
%
%
\begin{eqnarray}
\|w_t- \hat{w}_t \|^2 
&\leq&  \sum_{j=t-\tau}^{t-1} \| \eta_j d_{\xi_j}   (I-\Sigma_{t,j}) S^{\xi_j}_{u_j} \nabla   f(\hat{w}_j;\xi_j) \|^2  \nonumber  \\
&& + \sum_{i\neq j \in \{t-\tau, \ldots, t-1\}}
[\| \eta_j d_{\xi_j}  (I-\Sigma_{t,j})  S^{\xi_j}_{u_j}  \nabla  f(\hat{w}_j;\xi_j)\|^2\nonumber \\ &&+ \| \eta_i d_{\xi_i}   (I-\Sigma_{t,j})  S^{\xi_i}_{u_i}  \nabla  f(\hat{w}_i;\xi_i) \|^2]\sqrt{\Delta}/2 \nonumber  \\
&=& (1+\sqrt{\Delta}\tau) \sum_{j=t-\tau}^{t-1} \| \eta_j d_{\xi_j}    (I-\Sigma_{t,j}) S^{\xi_j}_{u_j}  \nabla  f(\hat{w}_j;\xi_j) \|^2  \nonumber  \\
&\leq&   (1+\sqrt{\Delta}\tau) \sum_{j=t-\tau}^{t-1} \eta_j^2  \|  d_{\xi_j}    S^{\xi_j}_{u_j}  \nabla  f(\hat{w}_j;\xi_j) \|^2 . \nonumber 
\end{eqnarray}
Taking expectations shows
$$\mathbb{E}[\|w_t- \hat{w}_t \|^2 ] \leq   (1+\sqrt{\Delta}\tau) \sum_{j=t-\tau}^{t-1} \eta_j^2  \mathbb{E}[ \| d_{\xi_j}   S^{\xi_j}_{u_j}  \nabla  f(\hat{w}_j;\xi_j) \|^2].$$
Now, we can apply Lemma \ref{lem:expect}: We first take the expectation over $u_j$ and this shows
$$\mathbb{E}[\|w_t- \hat{w}_t \|^2 ] \leq  (1+\sqrt{\Delta}\tau) \sum_{j=t-\tau}^{t-1} \eta_j^2 D \mathbb{E}[ \|\nabla f(\hat{w}_j;\xi_j) \|^2 ].$$
From Lemma \ref{lem_bounded_secondmoment_04} we infer
\begin{equation} \mathbb{E}[ \|\nabla f(\hat{w}_j;\xi_j) \|^2 ] \leq 4L\mathbb{E}[F(\hat{w}_j)-F(w_{*})] +N \label{LN}
\end{equation}
and by $L$-smoothness, see Equation \ref{eq:Lsmooth} with $\nabla F(w_*) =0$,
$$F(\hat{w}_j)-F(w_*) \leq \frac{L}{2} \| \hat{w}_j -w_{*} \|^2.$$
Combining the above inequalities proves the lemma.
\end{proof}

 Together with the next lemma we will be able to start deriving a recursive inequality from which we will be able to derive a bound on the convergence rate.

\begin{lem}\label{lemma:hogwild_1_new1}
Let $0 < \eta_t \leq \frac{1}{4LD}$  for all $t \geq 0$. Then,
\begin{align*}
 \mathbb{E}[\| w_{t+1} - w_{*} \|^2 |\mathcal{F}_t] \leq \left(1 - \frac{\mu \eta_t}{2} \right) \| w_{t} - w_{*} \|^2 + [(L+\mu)\eta_t + 2L^2 \eta_t^2D] \| \hat{w}_t - w_{t}  \|^2 + 2 \eta^2_t D N. 
\end{align*}
\end{lem}

\begin{proof}
Since $w_{t+1} = w_t - \eta_t d_{\xi_t}  S^{\xi_t}_{u_t} \nabla   f(\hat{w}_t;\xi_t)$, we have 
$$ \|w_{t+1}-w_{*}\|^2 = \|w_t-w_{*}\|^2 - 2\eta_t \langle d_{\xi_t}  S^{\xi_t}_{u_t} \nabla   f(\hat{w}_t;\xi_t), (w_t-w_{*})\rangle + \eta_t^2 \| d_{\xi_t}  S^{\xi_t}_{u_t}  \nabla  f(\hat{w}_t;\xi_t)\|^2.$$
We now take expectations over $u_t$ and $\xi_t$ and use Lemma \ref{lem:expect}:
\begin{eqnarray*}
&& \mathbb{E}[\| w_{t+1} - w_{*} \|^2 | \mathcal{F}_{t}] \\
&\leq &
   \| w_{t} - w_{*} \|^2 - 2 \eta_t \langle \nabla F(\hat{w}_t)  , (w_{t} - w_{*}) \rangle + \eta^2_t D \mathbb{E}[\|\nabla f(\hat{w}_t; \xi_t)\|^2 | \mathcal{F}_{t} ] \\
&=&  \| w_{t} - w_{*} \|^2 - 2 \eta_t \langle \nabla F(\hat{w}_t)  , (w_{t} - \hat{w}_t) \rangle - 2 \eta_t \langle \nabla F(\hat{w}_t)  , (\hat{w}_t - w_{*}) \rangle   \\
&&+ \eta^2_t D \mathbb{E}[\|\nabla f(\hat{w}_t; \xi_t)\|^2 | \mathcal{F}_{t} ].
\end{eqnarray*}
By \eqref{eq:stronglyconvex_00} and \eqref{eq:Lsmooth}, we have
\begin{align*}
- \langle \nabla F(\hat{w}_t)  , (\hat{w}_t - w_{*}) \rangle & \leq - [ F(\hat{w}_t) - F(w_{*}) ] - \frac{\mu}{2} \| \hat{w}_t - w_{*} \|^2, \mbox{ and} \tagthis \label{eq_newlemma_011} \\
- \langle \nabla F(\hat{w}_t)  , (w_{t} - \hat{w}_t) \rangle & \leq F(\hat{w}_t) - F(w_{t}) + \frac{L}{2} \| \hat{w}_t - w_{t}  \|^2  \tagthis \label{eq_newlemma_021} 
\end{align*}

Thus, $\mathbb{E}[\| w_{t+1} - w_{*} \|^2 | \mathcal{F}_{t}]$ is at most
\begin{align*}
& \overset{\eqref{eq_newlemma_011}\eqref{eq_newlemma_021}}{\leq} \| w_{t} - w_{*} \|^2 + 2 \eta_t [ F(\hat{w}_t) - F(w_{t})  ] + L\eta_t \| \hat{w}_t - w_{t}  \|^2  - 2 \eta_t[ F(\hat{w}_t ) - F(w_{*}) ]  \\ &\qquad  - \mu \eta_t \| \hat{w}_t - w_{*} \|^2+ \eta_t^2 D \mathbb{E}[\|\nabla f(\hat{w}_t; \xi_t)\|^2 | \mathcal{F}_{t} ] \\
&= \| w_{t} - w_{*} \|^2 - 2 \eta_t [ F(w_{t}) - F(w_{*})  ] + L\eta_t \| \hat{w}_t - w_{t}  \|^2  - \mu \eta_t \| \hat{w}_t - w_{*} \|^2  \\ &\qquad  + \eta^2_t D \mathbb{E}[\|\nabla f(\hat{w}_t; \xi_t)\|^2 | \mathcal{F}_{t} ].
\end{align*}

Since \begin{align*}
- \|\hat{w}_t - w_{*}\|^2 = - \|(w_{t} - w_{*}) - (w_{t} - \hat{w}_t) \|^2 \overset{\eqref{eq_aaa002}}{\leq} - \frac{1}{2} \| w_{t} - w_{*} \|^2 + \| w_{t} - \hat{w}_t \|^2, 
\end{align*}
$\mathbb{E}[\| w_{t+1} - w_{*} \|^2 | \mathcal{F}_{t}, \sigma_t] $ is at most 
$$
(1- \frac{\mu \eta_t}{2}) \| w_{t} - w_{*} \|^2 - 2 \eta_t [ F(w_{t}) - F(w_{*})  ] + (L+\mu)\eta_t \| \hat{w}_t- w_{t}  \|^2 
 + \eta^2_t D \mathbb{E}[\|\nabla f(\hat{w}_t; \xi_t)\|^2 | \mathcal{F}_{t} ].
$$

We now use $\| a \|^2 = \| a - b + b \|^2 \leq 2 \| a - b\|^2 + 2 \|b\|^2$ for $\mathbb{E}[\|\nabla f(\hat{w}_t; \xi_t)\|^2 | \mathcal{F}_{t} ]$ to obtain
\begin{equation}
\mathbb{E}[\|\nabla f(\hat{w}_t; \xi_t)\|^2 | \mathcal{F}_{t}] \leq 
 2 \mathbb{E}[\|\nabla f(\hat{w}_t; \xi_t)- \nabla f(w_{t}; \xi_t) \|^2 | \mathcal{F}_{t}] + 2 \mathbb{E}[\|\nabla f(w_{t}; \xi_t)\|^2 | \mathcal{F}_{t} ]. \label{ineq:Estuff}
\end{equation}
By Lemma \ref{lem_bounded_secondmoment_04}, we have
\begin{gather*}
\mathbb{E}[\|\nabla f(w_{t}; \xi_t)\|^2 | \mathcal{F}_{t}] \leq  4 L [ F(w_{t}) - F(w_{*}) ] + N. \tagthis \label{ineq:bounded_lemma3_hogwild}
\end{gather*} 
Applying \eqref{eq:Lsmooth_basic} twice\ gives 
$$\mathbb{E}[\|\nabla f(\hat{w}_t; \xi_t)- \nabla f(w_{t}; \xi_t) \|^2 | \mathcal{F}_{t},\sigma_t ]\leq L^2 \|\hat{w}_t - w_t \|^2$$
and  together with (\ref{ineq:Estuff}) and (\ref{ineq:bounded_lemma3_hogwild}) we obtain
$$\mathbb{E}[\|\nabla f(\hat{w}_t; \xi_t)\|^2 | \mathcal{F}_{t} ] \leq  2 L^2 \|\hat{w}_t - w_t \|^2 + 4 L [ F(w_{t}) - F(w_{*}) ] + N.$$
Plugging this into the previous derivation yields
 %
\begin{align*}
\mathbb{E}[\| w_{t+1} - w_{*} \|^2 | \mathcal{F}_{t}]& \leq 
(1- \frac{\mu \eta_t}{2}) \| w_{t} - w_{*} \|^2 - 2 \eta_t [ F(w_{t}) - F(w_{*})  ] + (L + \mu) \eta_t \| \hat{w}_{t} - w_{t}  \|^2 
 \\ &\qquad  + 2 L^2 \eta^2_t D \| \hat{w}_t - w_{t}  \|^2  + 8 L \eta^2_t D [F(w_{t}) - F(w_{*})] + 2\eta^2_t D N \\
&=(1- \frac{\mu \eta_t}{2}) \| w_{t} - w_{*} \|^2  + [(L+\mu)\eta_t + 2L^2\eta_t^2D] \| \hat{w}_t - w_{t}  \|^2  \\ &\qquad 
 - 2\eta_t( 1 - 4L\eta_tD) [F(w_{t}) - F(w_{*})] + 2\eta^2_t D N. 
\end{align*}
Since $\eta_t \leq \frac{1}{4LD}$, $-2\eta_t(1-4L\eta_tD) [F(w_{t}) - F(w_{*})]\leq 0$ (we can get a negative upper bound by applying strong convexity but this will not improve the asymptotic behavior of the convergence rate in our main result although it would improve the constant of the leading term making the final bound applied to SGD closer to the bound of Theorem \ref{thm_res_sublinear_new_02} for SGD),
$$\mathbb{E}[\| w_{t+1} - w_{*} \|^2 | \mathcal{F}_{t}] \leq \left(1 - \frac{\mu \eta_t}{2} \right) \| w_{t} - w_{*} \|^2 + [(L+\mu)\eta_t + 2L^2\eta_t^2D] \| \hat{w}_t - w_{t}  \|^2 + 2\eta^2_tD N$$ 
and this concludes the proof.
\end{proof} 

Assume 
$0 < \eta_t \leq \frac{1}{4LD}$  for all $t \geq 0$. Then,  after taking the full expectation of the inequality in  Lemma \ref{lemma:hogwild_1_new1}, we can
plug Lemma \ref{lemma:hogwild_21} into it which yields the recurrence
\begin{eqnarray} 
 \mathbb{E}[\| w_{t+1} - w_{*} \|^2]
&\leq& \left(1 - \frac{\mu \eta_t}{2} \right) \mathbb{E}[\| w_{t} - w_{*} \|^2] +  \nonumber  \\
&& [(L+\mu)\eta_t + 2L^2\eta_t^2D] (1+\sqrt{\Delta}\tau) D \sum_{j=t-\tau}^{t-1} \eta_j^2 (2L^2 \mathbb{E}[\| \hat{w}_j -w_{*} \|^2] +N) \nonumber \\ &&+ 2\eta^2_tD N. \label{eq:rec} \end{eqnarray}
 This can be solved by using the next lemma. For completeness, we follow the convention that an empty product is equal to 1 and an empty sum is equal to 0, i.e., 
\begin{gather*}
\prod_{i=h}^k g_i = 1 \text{ and } \sum_{i=h}^k g_i =0 \text{ if } k<h. \tagthis \label{eq:rec00} 
\end{gather*}

\begin{lem}
\label{lemma:hogwild_recursive_form}
Let $Y_t, \beta_t$ and $\gamma_t$ be sequences such that $Y_{t+1} \leq \beta_t Y_t + \gamma_t$, for all $t\geq 0$. Then,
\begin{equation}
 \label{hogwild:recursive_form}
Y_{t+1} \leq (\sum_{i=0}^t[\prod_{j=i+1}^t \beta_j]\gamma_i)+(\prod_{j=0}^t\beta_j)Y_0. 
\end{equation}
\end{lem}

\begin{proof}
We prove the lemma by using induction. It is obvious that \eqref{hogwild:recursive_form} is true for $t=0$ because $Y_1\leq \beta_1Y_0+ \gamma_1$. Assume as induction hypothesis that \eqref{hogwild:recursive_form} is true for $t-1$. Since $Y_{t+1} \leq \beta_t Y_t + \gamma_t$,
\begin{align*}
Y_{t+1} &\leq \beta_t Y_t + \gamma_t \\
&\leq \beta_{t}[(\sum_{i=0}^{t-1}[\prod_{j=i+1}^{t-1} \beta_j]\gamma_i)+(\prod_{j=0}^{t-1}\beta_j)Y_0] + \gamma_{t} \\
&\overset{\eqref{eq:rec00}}{=} (\sum_{i=0}^{t-1} \beta_{t}[\prod_{j=i+1}^{t-1} \beta_j]\gamma_i)+\beta_{t}(\prod_{j=0}^{t-1}\beta_j)Y_0 + (\prod_{j=t+1}^{t}\beta_j)\gamma_{t}\\
&= [(\sum_{i=0}^{t-1} [\prod_{j=i+1}^{t} \beta_j]\gamma_i)+ (\prod_{j=t+1}^{t}\beta_j)\gamma_{t}] +(\prod_{j=0}^{t}\beta_j)Y_0 \\
&= (\sum_{i=0}^{t}[\prod_{j=i+1}^{t} \beta_j]\gamma_i)+(\prod_{j=0}^{t}\beta_j)Y_0.
\end{align*}
\end{proof} 

Applying the above lemma to (\ref{eq:rec}) will yield the following bound.
 
\begin{lem}\label{lemma:hogwild_4_new1} 
Let $\eta_t = \frac{\alpha_t}{\mu(t+E)}$ with $4\leq \alpha_t \leq\alpha$ and $E = \max\{ 2\tau, \frac{4 L \alpha D}{\mu}\}$. Then, 
\begin{align*}
\mathbb{E}[\| w_{t+1} - w_{*} \|^2] & \leq \frac{\alpha^2D }{\mu^2} \frac{1}{(t + E - 1)^2} \cdot \\ & \qquad \cdot \left(\sum_{i=1}^t
 \left[ 4 a_i (1+\sqrt{\Delta}\tau)   [ N\tau + 2L^2 \sum_{j=i-\tau}^{i-1} \mathbb{E}[\| \hat{w}_j -w_* \|^2 ] +
 2 N \right]
\right)\\ &+ \frac{(E+1)^2}{(t + E - 1)^2} \mathbb{E}[\| w_{0} - w_{*} \|^2],
\end{align*} 
where $a_i = (L+\mu)\eta_i + 2L^2\eta_i^2D$.
\end{lem}

\begin{proof}
Notice that we may use   (\ref{eq:rec}) because  $\eta_t \leq \frac{1}{4LD}$ follows from $\eta_t = \frac{\alpha_t}{\mu(t+E)} \leq \frac{\alpha}{\mu(t+E)}$ combined with $E\geq \frac{4 L \alpha D}{\mu}$. From (\ref{eq:rec}) with $a_t = (L+\mu)\eta_t + 2L^2\eta_t^2D$  and $\eta_t$ being decreasing in $t$ we infer $ \mathbb{E}[\| w_{t+1} - w_{*} \|^2]$
\begin{eqnarray*}
 \leq&\left(1 - \frac{\mu \eta_t}{2} \right) \mathbb{E}[\| w_{t} - w_{*} \|^2] + a_t (1+\sqrt{\Delta}\tau) D \eta_{t-\tau}^2 \sum_{j=t-\tau}^{t-1}(2L^2\mathbb{E}[\| \hat{w}_j -w_{*} \|^2 ] +N) \\ &+ 2\eta^2_tD N \\
 =& \left(1 - \frac{\mu \eta_t}{2} \right) \mathbb{E}[\| w_{t} - w_{*} \|^2] + a_t (1+\sqrt{\Delta}\tau) D \eta_{t-\tau}^2 [ N\tau + 2L^2 \sum_{j=t-\tau}^{t-1} \mathbb{E}[\| \hat{w}_j -w_{*} \|^2 ]  \\ &+ 2\eta^2_tD N.
 \end{eqnarray*}
 
 Since $E\geq 2\tau$,  $\frac{1}{t-\tau+E}\leq \frac{2}{t+E}$. Hence, together with  $\eta_{t-\tau} = \frac{\alpha_{t-\tau}}{\mu(t-\tau+E)} \leq \frac{\alpha}{\mu(t-\tau+E)}$ we have
 \begin{equation} \eta_{t-\tau}^2 \leq \frac{4\alpha^2}{\mu^2}\frac{1}{(t+E)^2}. \label{etatR} \end{equation}
 This translates the above bound into
 \begin{align*}
\mathbb{E}[\| w_{t+1} - w_{*} \|^2] & \leq \beta_t \mathbb{E}[\| w_{t} - w_{*} \|^2] + \gamma_t,
\end{align*}
for 
\begin{align*}
\beta_t &= 1 - \frac{\mu \eta_t}{2}, \\
\gamma_t &=  4 a_t (1+\sqrt{\Delta}\tau)  D  \frac{\alpha^2}{\mu^2}\frac{1}{(t+E)^2} [ N\tau + 2L^2 \sum_{j=t-\tau}^{t-1} \mathbb{E}[\| \hat{w}_j -w_{*} \|^2 ] + 2\eta^2_tD N, where \\
a_t &=  (L+\mu)\eta_t + 2L^2\eta_t^2D.
\end{align*}

Application of Lemma \ref{lemma:hogwild_recursive_form} for $Y_{t+1} = \mathbb{E}[\| w_{t+1} - w_{*} \|^2]$ and $Y_{t} = \mathbb{E}[\| w_{t} - w_{*} \|^2]$ gives
\begin{align*}
\mathbb{E}[\| w_{t+1} - w_{*} \|^2] \leq \left(\sum_{i=0}^t \left[\prod_{j=i+1}^t \left(1 - \frac{\mu \eta_j}{2} \right) \right]\gamma_i \right)+\left(\prod_{j=0}^t \left(1 - \frac{\mu \eta_j}{2} \right) \right) \mathbb{E}[\| w_{0} - w_{*} \|^2]. 
\end{align*}
In order to analyze this formula, since $\eta_j = \frac{\alpha_j}{\mu(j + E)}$ with $\alpha_j\geq 4$, we have
\begin{align*}
1 - \frac{\mu \eta_j}{2} &= 1 - \frac{\alpha_j}{2(j+E)} \leq 1 - \frac{2}{j +E}, 
\end{align*}
Hence (we can also use $1-x\leq e^{-x}$ which leads to similar results and can be used to show that our choice for $\eta_t$ leads to the tightest convergence rates in our framework), 
\begin{align*}
\prod_{j=i}^t \left(1 - \frac{\mu \eta_j}{2} \right) &\leq \prod_{j=i}^t \left(1-\frac{2}{j+E}\right) =  \prod_{j=i}^t \frac{j+E-2}{j+E}\\
 &=\frac{i+E-2}{i+E}\frac{i+E-1}{i+E+1}\frac{i+E}{i+E+2}\frac{i+E+1}{i+E+3}\dotsc\frac{t+E-3}{t+E-1}\frac{t+E-2}{t+E} \\
 &=\frac{(i+E-2)(i+E-1)}{(t+E-1)(t+E)} \leq \frac{(i+E-1)^2}{(t+E-1)(t+E)} \leq \frac{(i+E)^2}{(t+E-1)^2}   .  
\end{align*}


From this calculation  we infer that
\begin{align*}
\mathbb{E}[\| w_{t+1} - w_{*} \|^2] \leq \left(\sum_{i=0}^t \left[ \frac{(i + E)^2}{(t + E - 1)^2} \right]\gamma_i \right)+ \frac{(E+1)^2}{(t + E - 1)^2} \mathbb{E}[\| w_{0} - w_{*} \|^2]. \tagthis \label{eq_newlemma02_041}
\end{align*}

Now, we substitute  $\eta_i \leq \frac{\alpha}{\mu(i + E)}$ in $\gamma_i$ and compute
\begin{align*}
&  \frac{(i + E)^2}{(t + E - 1)^2} \gamma_i \\
&=  \frac{(i + E)^2}{(t + E - 1)^2}
 4 a_i (1+\sqrt{\Delta}\tau)  D  \frac{\alpha^2}{\mu^2}\frac{1}{(i+E)^2} [ N\tau + 2L^2 \sum_{j=i-\tau}^{i-1} \mathbb{E}[\| \hat{w}_j -w_* \|^2 ] 
\\   &+  \frac{(i + E)^2}{(t + E - 1)^2} 2 N D \frac{\alpha^2}{\mu^2(i + E)^2} \\
 &= \frac{\alpha^2D }{\mu^2} \frac{1}{(t + E - 1)^2} \left[ 4 a_i (1+\sqrt{\Delta}\tau)   [ N\tau + 2L^2 \sum_{j=i-\tau}^{i-1} \mathbb{E}[\| \hat{w}_j -w_{*} \|^2 ] +
 2 N \right]. 
\end{align*}
Substituting this in \eqref{eq_newlemma02_041} proves the lemma. 
\end{proof}

As an immediate corollary we can apply the inequality $\|a+b\|^2\leq 2\|a\|^2+2\|b\|^2$ to $\mathbb{E}[\|\hat{w}_{t+1} - w_* \|^2]$ to obtain 
\begin{equation}
\label{eq:eq11111}
\mathbb{E}[\|\hat{w}_{t+1} - w_* \|^2] \leq 2\mathbb{E}[\|\hat{w}_{t+1} - w_{t+1} \|^2]
+2\mathbb{E}[\|w_{t+1} - w_*\|^2],
\end{equation}
which in turn can be bounded by the previous lemma together with Lemma \ref{lemma:hogwild_21}: $\mathbb{E}[\|\hat{w}_{t+1} - w_* \|^2]$
\begin{eqnarray*}
& \leq& 2(1+\sqrt{\Delta}\tau)  D \sum_{j=t+1-\tau}^{t} \eta_j^2 (2L^2 \mathbb{E}[\| \hat{w}_j -w_{*} \|^2] +N) + \\
&& 2\frac{\alpha^2D }{\mu^2} \frac{1}{(t + E - 1)^2} \left(\sum_{i=1}^t
 \left[ 4 a_i (1+\sqrt{\Delta}\tau)   [ N\tau + 2L^2 \sum_{j=i-\tau}^{i-1} \mathbb{E}[\| \hat{w}_j -w_{*} \|^2 ] +
 2 N \right]
\right)  + \\
&& \frac{(E+1)^2}{(t + E - 1)^2} \mathbb{E}[\| w_{0} - w_{*} \|^2]. 
\end{eqnarray*}

Now assume a decreasing sequence $Z_t$ for which we want to prove that $\mathbb{E}[\| \hat{w}_{t} - w_{*} \|^2]\leq Z_t$ by induction in $t$. Then, the above bound can be used together with the property that $Z_t$ and $\eta_t$ are decreasing in $t$ to show

\begin{eqnarray*}
\sum_{j=t+1-\tau}^{t} \eta_j^2 (2L^2 \mathbb{E}[\| \hat{w}_j -w_* \|^2] +N) &\leq& \tau  \eta_{t-\tau}^2 (2L^2 Z_{t+1-\tau} +N) \\  &\leq&
4\tau \frac{\alpha^2}{\mu^2}\frac{1}{(t+E-1)^2}  (2L^2 Z_{t+1-\tau} +N)
\end{eqnarray*}

where the last inequality follows from (\ref{etatR}), and
$$\sum_{j=i-\tau}^{i-1} \mathbb{E}[\| \hat{w}_j -w_{*} \|^2 ]\leq \tau Z_{i-\tau}.$$
From these inequalities we infer
\begin{eqnarray}
\mathbb{E}[\|\hat{w}_{t+1} - w_* \|^2]
& \leq& 8(1+\sqrt{\Delta}\tau) \tau D \frac{\alpha^2}{\mu^2}\frac{1}{(t+E-1)^2}  (2L^2 Z_{t+1-\tau} +N) + \nonumber \\
&& 2\frac{\alpha^2D }{\mu^2} \frac{1}{(t + E - 1)^2} \left(\sum_{i=1}^t
 \left[ 4 a_i (1+\sqrt{\Delta}\tau)  [ N\tau + 2L^2 \tau Z_{i-\tau}] +
 2 N \right]
\right)  + \nonumber \\
&& \frac{(E+1)^2}{(t + E - 1)^2} \mathbb{E}[\| w_{0} - w_{*} \|^2] . \label{bestbound}
\end{eqnarray}

Even if we assume a constant $Z\geq Z_0\geq Z_1\geq Z_2 \geq \ldots$, we can get a first bound on the convergence rate of vectors $\hat{w}^t$: Substituting $Z$ gives
\begin{eqnarray}
\mathbb{E}[\|\hat{w}_{t+1} - w_* \|^2]
& \leq& 8(1+\sqrt{\Delta}\tau) \tau D \frac{\alpha^2}{\mu^2}\frac{1}{(t+E-1)^2}  (2L^2 Z +N) + \nonumber \\
&& 2\frac{\alpha^2D }{\mu^2} \frac{1}{(t + E - 1)^2} \left(\sum_{i=1}^t
 \left[ 4 a_i (1+\sqrt{\Delta}\tau)   [ N\tau + 2L^2 \tau Z] +
 2 N \right]
\right)  + \nonumber \\
&& \frac{(E+1)^2}{(t + E - 1)^2} \mathbb{E}[\| w_{0} - w_{*} \|^2] . \label{eqforGen}
\end{eqnarray}

Since $a_i = (L+\mu)\eta_i + 2L^2\eta_i^2 D$ and $\eta_i \leq \frac{\alpha}{\mu(i + E)}$, we have
\begin{align*}
\sum_{i=1}^t a_i &= (L+\mu) \sum_{i=1}^t \eta_i + 2 L^2 D \sum_{i=1}^t \eta_i^2 \\
&\leq (L+\mu) \sum_{i=1}^t \frac{\alpha}{\mu(i + E)} + 2 L^2 D\sum_{i=1}^t \frac{\alpha^2}{\mu^2(i + E)^2} \\
&\leq \frac{(L+\mu)\alpha}{\mu} \sum_{i=1}^t \frac{1}{i} + \frac{2 L^2 \alpha^2D}{\mu^2} \sum_{i=1}^t \frac{1}{i^2} \\
&\leq \frac{(L+\mu)\alpha}{\mu} (1 + \ln t) + \frac{ L^2 \alpha^2 D\pi^2}{3 \mu^2}, \tagthis \label{eq_newlemma03_0211}
\end{align*}
where the last inequality is a property of the harmonic sequence $\sum_{i=1}^t \frac{1}{i}\leq 1 + \ln t$ and $\sum_{i=1}^t \frac{1}{i^2} \leq \sum_{i=1}^{\infty} \frac{1}{i^2} = \frac{\pi^2}{6}$. 

Substituting (\ref{eq_newlemma03_0211}) in (\ref{eqforGen}) and collecting terms yields
\begin{eqnarray}
&& \mathbb{E}[\|\hat{w}_{t+1} - w_* \|^2] \leq \nonumber \\
&&
2\frac{\alpha^2D }{\mu^2} \frac{1}{(t + E - 1)^2} \left(2N t + 
4(1+\sqrt{\Delta}\tau) \tau [N+2L^2Z] \left\{ \frac{(L+\mu)\alpha}{\mu} (1 + \ln t) + \frac{ L^2 \alpha^2 D\pi^2}{3 \mu^2 +1} \right\}
\right)  \nonumber  \\
&& + \frac{(E+1)^2}{(t + E - 1)^2} \mathbb{E}[\| w_{0} - w_{*} \|^2] .\label{refbound}
\end{eqnarray}
Notice that the asymptotic behavior in $t$ is dominated by the term
$$ \frac{4\alpha^2D N}{\mu^2} \frac{t}{(t + E - 1)^2}.$$
If we define $Z_{t+1}$ to be the right hand side of (\ref{refbound}) and observe that this $Z_{t+1}$ is decreasing and a constant $Z$ exists (since the terms with $Z$ decrease much faster in $t$ compared to the dominating term), then  this $Z_{t+1}$ satisfies the derivations done above  and a proof by induction can be completed. 

Our derivations prove our main result: The expected convergence rate of read vectors is 
 \begin{eqnarray*}
\mathbb{E}[\|\hat{w}_{t+1} - w_* \|^2]
& \leq&  \frac{4\alpha^2D N}{\mu^2} \frac{t}{(t + E - 1)^2} + O\left(\frac{\ln t}{(t+E-1)^{2}}\right).
\end{eqnarray*}
We can use this result in Lemma \ref{lemma:hogwild_4_new1} in order to show that
  the expected convergence rate $\mathbb{E}[\|w_{t+1} - w_* \|^2]$ satisfies the same bound. 

We remind the reader, that in the $(t+1)$-th iteration at most $\leq \lceil |D_{\xi_t}|/D \rceil$ vector positions are updated. Therefore the expected number of single vector entry updates is at most $\bar{\Delta}_D /D$.



\vspace{.3cm}

\noindent
\textbf{Theorem~\ref{theorem:Hogwild_newnew1}}. 
 \textit{Suppose Assumptions \ref{ass_stronglyconvex}, \ref{ass_smooth}, \ref{ass_convex} and \ref{ass_tau}  and  consider Algorithm~\ref{HogWildAlgorithm}. Let  $\eta_t = \frac{\alpha_t}{\mu(t+E)}$ with $4\leq \alpha_t \leq\alpha$ and $E = \max\{ 2\tau, \frac{4 L \alpha D}{\mu}\}$. Then, $t' = t \bar{\Delta}_D /D$ is the expected number of single vector entry updates after $t$ iterations and
  \begin{eqnarray*}
 \mathbb{E}[\|\hat{w}_{t} - w_* \|^2] &\leq \frac{4\alpha^2D N}{\mu^2} \frac{t}{(t + E - 1)^2} + O\left(\frac{\ln t}{(t+E-1)^{2}}\right),\\
 \mathbb{E}[\|w_{t} - w_* \|^2]  &\leq \frac{4\alpha^2D N}{\mu^2} \frac{t}{(t + E - 1)^2} + O\left(\frac{\ln t}{(t+E-1)^{2}}\right),
\end{eqnarray*} 
where $N = 2 \mathbb{E}[ \|\nabla f(w_{*}; \xi)\|^2 ]$ and $w_{*} = \arg \min_w F(w)$. 
}

\subsection{Convergence without Convexity of Component Functions}\label{subsec_withoutconvex}
 
 For the non-convex case of the component functions, $L$ in (\ref{LN}) must be replaced by $L\kappa$ and as a result $L^2$ in Lemma \ref{lemma:hogwild_21} must be replaced by $L^2\kappa$.  Also $L$ in (\ref{ineq:bounded_lemma3_hogwild}) must be replaced by $L\kappa$. We now require that $\eta_t \leq \frac{1}{4L\kappa D}$ so that $-2\eta_t(1-4L\kappa \eta_tD) [F(w_{t}) - F(w_{*})]\leq 0$. This leads to Lemma \ref{lemma:hogwild_1_new1} where no changes are needed except requiring  $\eta_t\leq \frac{1}{4L\kappa D}$. The changes in Lemmas \ref{lemma:hogwild_21} and \ref{lemma:hogwild_1_new1}  lead to a Lemma \ref{lemma:hogwild_4_new1} where we require $E\geq \frac{4L\kappa \alpha D}{\mu}$ and where in the bound of the expectation $L^2$ must be replaced by $L^2\kappa$. This perculates through to inequality (\ref{refbound})  with a similar change finally leading to Theorem \ref{thm_6}, i.e., Theorem \ref{theorem:Hogwild_newnew1} where we only need to strengthen the condition on $E$ to $E\geq \frac{4L\kappa \alpha D}{\mu}$ in order to remove Assumption \ref{ass_convex}.
 
 \subsection{Sensitivity to $\tau$}\label{subsec_sens_tau}
 
 What about the upper bound's sensitivity with respect to $\tau$?
 Suppose $\tau$ is not a constant but an increasing function of $t$, which also makes $E$ a function of $t$:
 $$ \frac{2 L \alpha D}{\mu}\leq \tau(t) \leq t  \mbox{ and } E(t)=2\tau(t).$$
 In order to obtain a similar theorem we  increase the lower bound on $\alpha_t$ to 
 $$ 12\leq \alpha_t \leq \alpha.$$
 This allows us to modify the proof of Lemma \ref{lemma:hogwild_4_new1} where we analyse the product
 $$ \prod_{j=i}^t \left(1 - \frac{\mu \eta_j}{2} \right).$$
 Since $\alpha_j\geq 12$ and $E(j)=2\tau(j)\leq 2j$,
 $$ 1-\frac{\mu \eta_j}{2} = 1 - \frac{\alpha_j}{2(j+E(j))}\leq 1- \frac{12}{2(j+2j)}=1-\frac{2}{j}\leq 1-\frac{2}{j+1}.$$
 The remaining part of the proof of Lemma \ref{lemma:hogwild_4_new1} continues as before where constant $E$ in the proof is replaced by $1$. This yields
 instead of (\ref{eq_newlemma02_041})  \begin{align*}
\mathbb{E}[\| w_{t+1} - w_{*} \|^2] \leq \left(\sum_{i=1}^t \left[ \frac{(i+1)^2}{t^2} \right]\gamma_i \right)+ \frac{4}{t^2} \mathbb{E}[\| w_{0} - w_{*} \|^2]. 
\end{align*}
We again substitute  $\eta_i \leq \frac{\alpha}{\mu(i + E(i))}$ in $\gamma_i$, realize that $\frac{(i+1)}{(i+E(i))}\leq 1$, and compute
\begin{align*}
& \frac{(i + 1)^2}{t^2} \gamma_i \\
 &=  \frac{(i + 1)^2}{t^2}
 4 a_i (1+\sqrt{\Delta}\tau(i)) D  \frac{\alpha^2}{\mu^2}\frac{1}{(i+E(i))^2} [ N\tau(i) + 2L^2 \sum_{j=i-\tau(i)}^{i-1} \mathbb{E}[\| \hat{w}_j -w_* \|^2 ] 
  \\ &+  \frac{(i + 1)^2}{t^2} 2 N D \frac{\alpha^2}{\mu^2(i + E(i))^2} \\
 &\leq \frac{\alpha^2D }{\mu^2} \frac{1}{t^2} \left[ 4 a_i (1+\sqrt{\Delta}\tau(i))  [ N\tau(i) + 2L^2 \sum_{j=i-\tau(i)}^{i-1} \mathbb{E}[\| \hat{w}_j -w_{*} \|^2 ] +
 2 N \right]. 
\end{align*}
This gives a new Lemma  \ref{lemma:hogwild_4_new1}:

\begin{lem}\label{lemma:hogwild_4_new1_Gen} Assume  $\frac{2 L \alpha D}{\mu}\leq \tau(t) \leq t$ with $\tau(t)$ monotonic increasing.
Let $\eta_t = \frac{\alpha_t}{\mu(t+E(t))}$ with $12\leq \alpha_t \leq\alpha$ and $E(t) = 2\tau(t)$. Then, 
\begin{align*}
\mathbb{E}[\| w_{t+1} - w_{*} \|^2] & \leq 
 \frac{\alpha^2D }{\mu^2} \frac{1}{t^2} \cdot \\ & \qquad \cdot \left(\sum_{i=1}^t
 \left[ 4 a_i (1+\sqrt{\Delta}\tau(i))  [ N\tau(i) + 2L^2 \sum_{j=i-\tau(i)}^{i-1} \mathbb{E}[\| \hat{w}_j -w_* \|^2 ] +
 2 N \right]
\right) \\ & \quad + \frac{4}{t^2} \mathbb{E}[\| w_{0} - w_{*} \|^2],
\end{align*} 
where $a_i = (L+\mu)\eta_i + 2L^2\eta_i^2D$.
\end{lem}

Now we can continue the same analysis that led to Theorem \ref{theorem:Hogwild_newnew1} and conclude that there exists a constant $Z$ such that, see (\ref{eqforGen}),
\begin{eqnarray}
\mathbb{E}[\|\hat{w}_{t+1} - w_* \|^2]
& \leq& 8 (1+\sqrt{\Delta}\tau(t))\tau(t) D \frac{\alpha^2}{\mu^2}\frac{1}{t^2}  (2L^2 Z +N) + \nonumber \\
&& 2\frac{\alpha^2D }{\mu^2} \frac{1}{t^2} \left(\sum_{i=1}^t
 \left[ 4 a_i  (1+\sqrt{\Delta}\tau(i))  [ N\tau(i) + 2L^2 \tau(i) Z] +
 2 N \right]
\right)  + \nonumber \\
&& \frac{4}{t^2} \mathbb{E}[\| w_{0} - w_{*} \|^2] . \label{eqderGen1}
\end{eqnarray}

Let us assume
\begin{equation}
\tau(t)\leq \sqrt{t \cdot L(t)}, \label{eqR1}
\end{equation}
where 
$$L(t)= \frac{1}{\ln t} - \frac{1}{(\ln t)^2}$$
which has the property that the derivative of $t/(\ln t)$ is equal to $L(t)$.
Now we observe
\begin{eqnarray*}
 \sum_{i=1}^t a_i \tau(i)^2 &=& \sum_{i=1}^t [(L+\mu)\eta_i + 2L^2\eta_i^2D]\tau(i)^2
\leq  \sum_{i=1}^t [(L+\mu)\frac{\alpha}{\mu i} + 2L^2\frac{\alpha^2}{\mu^2 i^2}D]\cdot i L(i) \\
&=& \frac{(L+\mu)\alpha}{\mu}  \sum_{i=1}^t  L(i) + O(\ln t)
= \frac{(L+\mu)\alpha}{\mu}  \frac{t}{\ln t} + O(\ln t) 
\end{eqnarray*}
and
\begin{eqnarray*}
 \sum_{i=1}^t a_i \tau(i) &=& \sum_{i=1}^t [(L+\mu)\eta_i + 2L^2\eta_i^2D]\tau(i)
\leq  \sum_{i=1}^t [(L+\mu)\frac{\alpha}{\mu i} + 2L^2\frac{\alpha^2}{\mu^2 i^2}D]\cdot \sqrt{i} \\
&=& O(\sum_{i=1}^t  \frac{1}{\sqrt{i}})
=  O(\sqrt{t}) .
\end{eqnarray*}

Substituting both inequalities in  (\ref{eqderGen1}) gives
$\mathbb{E}[\|\hat{w}_{t+1} - w_* \|^2]$
\begin{eqnarray}
& \leq& 8  (1+\sqrt{\Delta}\tau(t))\tau(t) D \frac{\alpha^2}{\mu^2}\frac{1}{t^2}  (2L^2 Z +N) + \nonumber \\
&& 2\frac{\alpha^2D }{\mu^2} \frac{1}{t^2} \left(2Nt + 4\sqrt{\Delta}[  \frac{(L+\mu)\alpha}{\mu} \frac{t}{\ln t} + O(\ln t) ]  [ N + 2L^2 Z] +O(\sqrt{t})
\right)  + \nonumber \\
&& \frac{4}{t^2} \mathbb{E}[\| w_{0} - w_{*} \|^2]  \nonumber \\
&\leq &
 2\frac{\alpha^2D }{\mu^2} \frac{1}{t^2} \left(2Nt + 4\sqrt{\Delta}[ (1+ \frac{(L+\mu)\alpha}{\mu}) \frac{t}{\ln t} + O(\ln t) ]  [ N + 2L^2 Z] +O(\sqrt{t})
\right)  \nonumber \\
&& +\frac{4}{t^2} \mathbb{E}[\| w_{0} - w_{*} \|^2]   
    \label{ineqterms}
\end{eqnarray}
Again we define $Z_{t+1}$ as the right hand side of this inequality.  Notice that $Z_t = O(1/t)$, since the above derivation
 proves
$$ \mathbb{E}[\|\hat{w}_{t+1} - w_* \|^2] \leq  \frac{4\alpha^2D N}{\mu^2} \frac{1}{t} + O(\frac{1}{t\ln t}).$$

Summarizing we have the following main lemma:

\begin{lem} \label{lemtau}
Let Assumptions \ref{ass_stronglyconvex}, \ref{ass_smooth}, \ref{ass_convex} and \ref{ass_tau} hold and  consider Algorithm~\ref{HogWildAlgorithm}. 
Assume  $\frac{2 L \alpha D}{\mu}\leq \tau(t)\leq \sqrt{t \cdot L(t)}$ with $\tau(t)$ monotonic increasing.
Let  $\eta_t = \frac{\alpha_t}{\mu(t+2\tau(t))}$ with $12\leq \alpha_t \leq\alpha$. Then,
 the expected convergence rate of read vectors is 
 \begin{eqnarray*}
\mathbb{E}[\|\hat{w}_{t+1} - w_* \|^2]
& \leq& 
 \frac{4\alpha^2D N}{\mu^2} \frac{1}{t} + O(\frac{1}{t\ln t}),
\end{eqnarray*}
 where $L(t)=\frac{1}{\ln t} - \frac{1}{(\ln t)^2}$. The expected convergence rate $\mathbb{E}[\|w_{t+1} - w_* \|^2]$ satisfies the same bound.
 \end{lem}

Notice that we can plug $Z_t = O(1/t)$ back into an equivalent of (\ref{bestbound}) where we may bound $Z_{i-\tau(i)}=O(1/(i-\tau(i))$ which replaces $Z$ in the second line of (\ref{eqforGen}). On careful examination this leads to a new upper bound (\ref{ineqterms}) where the $2L^2Z$ terms gets absorped in a higher order term.
This can be used to show that, for
$$ t\geq T_0 = \exp[ 2\sqrt{\Delta}(1+\frac{(L+\mu)\alpha}{\mu})],$$
the higher order terms that contain $\tau(t)$ (as defined above) are at most the leading term as given in Lemma \ref{lemtau}.


 Upper bound (\ref{ineqterms}) also  shows that, for  
$$ t \geq T_1 = \frac{\mu^2}{\alpha^2 N D}\|w_0-w_*\|^2,$$
the higher order term that contains $\|w_0-w_*\|^2$ is at most the leading term.

\subsection{Convergence of Hogwild! with probability 1}
\label{subsec:convergence_Hogwild_wp1}

\begin{lem}
Let us consider the sequence $w_0,w_1,w_2,\dotsc,w_t,\dotsc,w_n$ generated by \eqref{eqw}:
$$
 w_{t+1} = w_t - \eta_t d_{\xi_t}  S^{\xi_t}_{u_t} \nabla f(\hat{w}_t;\xi_t), 
$$
and define 
$$
m_t= \max_{0 \leq i \leq n, 0 \leq t' \leq t} \| \nabla f(w_{t'};\xi_i) \|.
$$
Then, 
\begin{align*}
 m_t\leq m_0 \exp(L D \sum_{i=0}^{t-1} \eta_i), \tagthis \label{eq:mt_formula1}
\end{align*} 
where $d_{\xi_t} \leq D$ for all $\xi_t$. 
\end{lem}

\begin{proof}
From the $L$-smoothness assumption (i.e., $\| \nabla f(w;\xi) - \nabla f(w';\xi) \| \leq L \| w - w' \|, \ \forall w,w' \in \mathbb{R}^d$), for any given $\xi_i$:

\begin{align*}
\| \nabla f(w_{t+1};\xi_i) - \nabla f(w_t;\xi_i) \| &\leq L \| w_{t+1} - w_t \| = L \| \eta_t d_{\xi_t}  S^{\xi_t}_{u_t} \nabla f(\hat{w}_t;\xi_t) \| \\
&\leq LD\eta_t\| \nabla f(\hat{w}_t;\xi_t) \|. 
\end{align*}

Since 
$$
m_t= \max_{0 \leq i \leq n, 0 \leq t' \leq t} \| \nabla f(w_{t'};\xi_i) \|,
$$
we have 
$$
\| \nabla f(w_{t+1};\xi_i) - \nabla f(w_t;\xi_i) \| \leq L D \eta_t m_t 
$$
for any $i \in [n]$ and $t$.
Using the triangular inequality, we obtain
\begin{align*}
 \| \nabla f(w_{t+1};\xi_i) \| &\leq  \| \nabla f(w_{t};\xi_i) \| +  \| \nabla f(w_{t+1};\xi_i) - \nabla f(w_{t};\xi_i)  \|  \\
  &\overset{\forall i, \| \nabla f(w_t;\xi_i) \| \leq m_t}{\leq} m_t + LD \eta_t \| \nabla f(\hat{w}_t;\xi_t) \| \\    
   &\overset{\forall i, \| \nabla f_i(\hat{w}_t;\xi_i) \| \leq m_t}{\leq} (1+LD\eta_t) m_t. 
\end{align*}

Moreover, the result above implies $m_{t+1} \leq (1+LD \eta_t) m_t$ and unrolling $m_{t}$ yields
\begin{align*}
m_{t+1} \leq m_0  \prod_{i=0}^t (1+LD \eta_i).
\end{align*}
For all $x\geq 0$, it is always true that $1+x \leq \exp(x)$. Hence, we have
\begin{align*}
 m_{t+1}\leq m_0 \prod_{i=0}^t (1+LD \eta_i) \leq m_0  \prod_{i=0}^t \exp(LD \eta_i) \leq m_0 \exp(LD [\sum_{i=0}^t \eta_i]).
\end{align*}
\end{proof}

\noindent
\textbf{Theorem~\ref{Hogwild:theorem_convergence} (Sufficient conditions for almost sure convergence for Hogwild!)}
\textit{Let Assumptions \ref{ass_stronglyconvex}, \ref{ass_smooth}, \ref{ass_convex} and \ref{ass_tau} hold. Consider Hogwild! method described in Algorithm~\ref{HogWildAlgorithm} with a stepsize sequence such that}
\begin{align*}
0 < \eta_t=\frac{1}{LD(2+\beta)(k+t)} < \frac{1}{4LD} , \beta>0, k \geq 3\tau. 
\end{align*}
\textit{Then, the following holds w.p.1 (almost surely)}
\begin{align*}
\|w_{t} - w_{*} \| \rightarrow 0. 
\end{align*} 

\begin{proof}
As shown in Lemma~\ref{lemma:hogwild_1_new1}, for $0 < \eta_t \leq \frac{1}{4LD}$, we have
\begin{align*}
 \mathbb{E}[\| w_{t+1} - w_{*} \|^2 |\mathcal{F}_t] &\leq \left(1 - \frac{\mu \eta_t}{2} \right) \| w_{t} - w_{*} \|^2 + [(L+\mu)\eta_t + 2L^2 \eta_t^2D] \| \hat{w}_t - w_{t}  \|^2 + 2 \eta^2_t D N\\
 &= \| w_{t} - w_{*} \|^2  - \frac{\mu \eta_t}{2} \| w_{t} - w_{*} \|^2 + [(L+\mu)\eta_t + 2L^2 \eta_t^2D] \| \hat{w}_t - w_{t}  \|^2\\& + 2 \eta^2_t D N. 
\end{align*}

If we can show that $\sum_{t=0}^\infty [(L+\mu)\eta_t + 2L^2 \eta_t^2D] \| \hat{w}_t - w_{t}  \|^2 $ is finite, then it is straight forward to apply the proof technique from Theorem~\ref{thm_general_02_new_02} to show that $\|w_t - w_* \|^2 \rightarrow 0$ w.p.1. From the proof of Lemma~\ref{lemma:hogwild_21}, we know $\|w_t- \hat{w}_t \|^2$ is at most 

\begin{align*}
(1+\sqrt{\Delta}\tau) \sum_{j=t-\tau}^{t-1} \eta_j^2  \|  d_{\xi_j}    S^{\xi_j}_{u_j}  \nabla  f(\hat{w}_j;\xi_j) \|^2&\leq (1+\sqrt{\Delta}\tau)D^2 \sum_{j=t-\tau}^{t-1} \eta_j^2  \|  \nabla  f(\hat{w}_j;\xi_j) \|^2\\&\leq (1+\sqrt{\Delta}\tau)D^2 \tau m^2_t \eta_{t-\tau}^2 . 
\end{align*}

Since $\eta_{t-\tau}=(1-\frac{\tau}{k+t-\tau})\eta_t < \frac{1}{2}\eta_t$ when $k \geq 3\tau$ for all $t\geq 0$, it yields $\|w_t- \hat{w}_t \|^2 < (1+\sqrt{\Delta}\tau)D^2 \tau \frac{1}{4} \eta^2_t m^2_t$. Hence $\sum_{t=0}^\infty [(L+\mu)\eta_t + 2L^2 \eta_t^2D] \| \hat{w}_t - w_{t}  \|^2 $ is at most
$$
    [(L+\mu)\eta_0 + 2L^2 \eta_0^2D] (1+\sqrt{\Delta}\tau)D^2 \tau \sum_{t=0}^\infty \eta^2_t m^2_t.
$$

Combining $m_t\leq m_0 \exp(LD \sum_{i=0}^t \eta_i)$ (see~\eqref{eq:mt_formula1}) and  
$\eta_i= \frac{1}{LD(2+\beta)(k+i)}$ yields
\begin{align*}
m_t &\leq m_0 \exp(\frac{1}{2+\beta}\sum_{i=1}^t\frac{1}{i}) \leq  m_0 \exp(\frac{1}{2+\beta}(1+\ln t)) \leq m_0 \exp(\frac{1}{2+\beta})t^{\frac{1}{2+\beta}}.
\end{align*}
The second inequality is a property of harmonic number $H_t =\sum_{i=1}^t \frac{1}{i} \leq 1 + \ln t$. 
Hence, $$\eta_t m_t \leq \frac{1}{L(2+\beta)t} m_0 \exp(\frac{1}{2+\beta})t^{\frac{1}{2+\beta}} \leq \frac{m_0 \exp(\frac{1}{2+\beta})}{L(2+\beta)} t^{\frac{-(1+\beta)}{2+\beta}}.$$
Hence, we obtain 
$$
(\eta_t m_t)^2 \leq   [\frac{m_0  \exp(\frac{1}{2+\beta})}{L(2+\beta)}]^2 t^{\frac{-(2+2\beta)}{2+\beta}} \leq  [\frac{m_0 \exp(\frac{1}{2+\beta})}{L(2+\beta)}]^2 t^{-(1+\rho)}, 
$$
where $\rho = \frac{\beta}{2+\beta}$. 
Due to the property of over-harmonic series, $\sum_{t=1}^\infty \frac{1}{t^{1+\rho}}$ converges for any $\rho > 0$. In other words, $\sum_{t=0}^\infty (\eta_t m_t)^2$ is finite or $\sum_{t=0}^\infty [(L+\mu)\eta_t + 2L^2 \eta_t^2D] \| \hat{w}_t - w_{t}  \|^2 $ is finite.
\end{proof} 
\subsection{Convergence of Large Stepsizes}
\label{subsec:convergence_large_stepsize}

\textbf{Theorem~\ref{theore:SGD_convergence}}
\textit{Let Assumptions \ref{ass_stronglyconvex}, \ref{ass_smooth}, and \ref{ass_convex} hold. Consider Algorithm \ref{sgd_algorithm} with a stepsize sequence such that
$\eta_t \leq \frac{1}{2L}$, $\eta_t \rightarrow 0$, $\frac{d}{dt}\eta_t \leq 0$
and $\sum_{t=0}^\infty \eta_t \rightarrow \infty$.
} 
\textit{Then,} $$\mathbb{E}[\| w_{t+1} - w_{*} \|^2 ] \rightarrow 0.$$

\begin{proof}
As shown in~\eqref{main_ineq_sgd_new02}
$$\mathbb{E}[\| w_{t+1} - w_{*} \|^2 ]  \leq (1-\mu \eta_t) \mathbb{E}[ \| w_{t} - w_{*} \|^2 ] + \eta_t^2 N,$$
when $\eta_t \leq \frac{1}{2L}$.

Let $Y_{t+1}=\mathbb{E}[\| w_{t+1} - w_{*} \|^2 ]$, $Y_{t}=\mathbb{E}[\| w_{t} - w_{*} \|^2 ]$, $\beta_t= 1 - \mu \eta_t$ and $\gamma_t=\eta^2_tN$. As proved in Lemma~\ref{lemma:hogwild_recursive_form}, if $Y_{t+1} \leq \beta_t Y_t + \gamma_t$, then 
\begin{align*}
Y_{t+1} &\leq \sum_{i=0}^t [\prod_{j=i+1}^t \beta_j] \gamma_i + (\prod_{i=0}^t\beta_i)Y_1\\
&=  \sum_{i=0}^t [\prod_{j=i+1}^t (1 - \mu \eta_j)] \gamma_i + [\prod_{i=0}^t(1 - \mu \eta_i)]\mathbb{E}[\| w_{1} - w_{*} \|^2 ]\\
\end{align*}

Let us define 
\begin{equation}
\label{eq:nj_etaj}
n(j)=\mu \eta_j. 
\end{equation}
Since $1-x \leq \exp(-x)$ for all $x\geq 0$, $$\prod_{j=i+1}^t (1 - \mu \eta_j)\leq \exp(-\sum_{j=i+1}^t(\mu \eta_j))= \exp(-\sum_{j=i+1}^t n(j)).$$ 
Furthermore, since $n(j)$ is decreasing in $j$, we have
$$\sum_{j=i+1}^t n(j) \geq \int_{x=i+1}^{t+1}n(x)dx.$$
These two inequalities can be used to derive
\begin{align*}
Y_{t+1} &\leq \sum_{i=0}^t \exp(-\sum_{j=i+1}^t n(j)) n(i)^2 N + \exp(-\sum_{j=0}^t n(j)) Y_0 \\
&\leq \sum_{i=0}^t \exp(-\int_{x=i+1}^{t+1}n(x)dx) n(i)^2 N + \exp(-\int_{x=0}^{t+1}n(x)dx) Y_0 \\
&=\sum_{i=0}^t \exp(-[M(t+1)-M(i+1)]) n(i)^2 N + \exp(-M(t+1)) Y_0,
\end{align*}
where $$M(y)=\int_{x=0}^y n(x)dx \text{ and } \frac{d}{dy}M(y) = n(y).$$ 

We focus on $$F=\sum_{i=0}^t \exp(-[M(t+1)-M(i+1)]) n(i)^2.$$
We notice that
$$F=\exp(-M(t+1)) \sum_{i=0}^t \exp(M(i+1)) n(i)^2.$$
We know that $\exp(M(x+1))$ increases and $n(x)^2$ decreases, hence, in the most general case either their product first decreases and then starts to increase or their product keeps on increasing. We first discuss the decreasing and increasing case.
  Let $a(x)=\exp(M(x+1))n(x)^2$ denote this product and let integer $j\geq 0$ be such that $a(0)\geq a(1) \geq \ldots \geq a(j)$ and $a(j)\leq a(j+1)\leq a(j+2)\leq \ldots$ (notice that $j=0$ expresses the situation where $a(i)$ only increases). Function $a(x)$ for $x\geq 0$ is minimized for some value $h$ in $[j,j+1)$. For $1\leq i\leq j$, $a(i)\leq \int_{x=i-1}^{i} a(x) dx$, and for $j+1\leq i$, $a(i)\leq \int_{x=i}^{i+1} a(x) dx$.
 This yields the upper bound 
 \begin{align*}
 \sum_{i=0}^t a(i) &= a(0) + \sum_{i=1}^{j} a(i) + \sum_{i=j+1}^t a(i) \\
 &\leq a(0) + \int_{x=0}^{j} a(x) dx + \int_{x=j+1}^{t+1} a(x) dx, \\
 &\leq a(0) + \int_{x=0}^{t+1} a(x) dx.
 \end{align*}
 The same upper bound holds for the other case as well, i.e., if $a(i)$ is only decreasing.
 We conclude 
 $$F \leq \exp(-M(t+1)) [\exp(M(2))n(0)^2 + \int_{x=0}^{j+1} \exp(M(x+1))n(x)^2 dx ] .$$
 Combined with
  $$M(x+1)=\int_{y=0}^{x+1}n(y)dy\leq \int_{y=0}^{x}n(y)dy + n(x)=M(x) +n(x) $$
  we obtain
   \begin{align*}
   F &\leq \exp(-M(t+1)) [\exp(M(2))n(0)^2 + \int_{x=0}^{t+1} \exp(M(x))n(x)^2 \exp(n(x)) dx ]\\
   &\leq  \exp(-M(t+1)) [\exp(M(2))n(0)^2 + \exp(n(0)) \int_{x=0}^{t+1} \exp(M(x))n(x)^2  dx ].
   \end{align*}
   This gives
   \begin{align*}
Y_{t+1} &\leq  \exp(-M(t+1)) [\exp(M(1))n(0)^2 + \exp(n(0)) \int_{x=0}^{t+1} \exp(M(x))n(x)^2  dx ]  N \\
& + \exp(-M(t+1)) Y_0 \\
&=  N \exp(n(0)) C(t+1)  + \exp(-M(t+1)) [\exp(M(1))n(0)^2N  + Y_0], \tagthis \label{eq:Yt_Ct_nt}
\end{align*}
where
$$ C(t) = \exp(-M(t)) \int_{x=0}^{t} \exp(M(x))n(x)^2  dx .$$

 For $y\leq t$, we derive (notice that $n(x)$ is decreasing)
\begin{align*}
C(t) &= \exp(-M(t)) \int_{x=0}^{t} \exp(M(x))n(x)^2  dx \\
&= \exp(-M(t)) \int_{x=0}^{y} \exp(M(x))n(x)^2  dx + \exp(-M(t)) \int_{x=y}^{t} \exp(M(x))n(x)^2  dx \\
&\leq \exp(-M(t)) \int_{x=0}^{y} \exp(M(x))n(x)^2  dx + \exp(-M(t)) \int_{x=y}^{t} \exp(M(x))n(x)n(y)  dx \\
&= \exp(-M(t)) \int_{x=0}^{y} \exp(M(x))n(x)^2  dx + \exp(-M(t))n(y) \int_{x=y}^t \exp(M(x))n(x)dx \\
&= \exp(-M(t)) \int_{x=0}^{y} \exp(M(x))n(x)^2  dx + n(y)[1-\exp(-M(t))\exp(M(y))] \\
&\leq \exp(-M(t)) \int_{x=0}^{y} \exp(M(x))n(x)^2  dx + n(y).
\end{align*}

Let $\epsilon>0$. Since $n(y) \rightarrow 0$ as $y\rightarrow \infty$, there exists a $y$ such that $n(y)\leq \epsilon/2$.
Since $M(t)\rightarrow \infty$ as $t\rightarrow \infty$, $\exp(-M(t))\rightarrow 0$ as $t\rightarrow \infty$. Hence, there exists a $T$ such that for $t\geq T$,
$\exp(-M(t)) \int_{x=0}^{y} \exp(M(x))n(x)^2  dx \leq \epsilon/2$. This implies $C(t)\leq \epsilon$ for $t\geq T$.
This proves $C(t)\rightarrow 0$ as $t\rightarrow \infty$, and we conclude $Y_t\rightarrow 0$ as  $t\rightarrow \infty$.
\end{proof}

\noindent
\textbf{Theorem~\ref{theorem:convergence_rate}}
\textit{Let Assumptions \ref{ass_stronglyconvex}, \ref{ass_smooth}, and \ref{ass_convex} hold. Consider Algorithm \ref{sgd_algorithm} with a stepsize sequence such that $\eta_t \leq \frac{1}{2L}$, $\eta_t \rightarrow 0$, $\frac{d}{dt}\eta_t \leq 0$, and $\sum_{t=0}^\infty \eta_t \rightarrow \infty$. Then,} 
\begin{align*}
    \mathbb{E}[\| w_{t+1} - w_{*} \|^2 ] & \leq N\exp(n(0)) 2n(M^{-1}(\ln [ \frac{n(t+1)}{n(0)} ]+ M(t+1))) \\
        & \quad + \exp(-M(t+1)) [\exp(M(1))n(0)^2  N  + \mathbb{E}[\| w_{0} - w_{*} \|^2 ]],
\end{align*}
\textit{where} $n(t)=\mu \eta_t$ \textit{and} $M(t)=\int_{x=0}^t n(x)dx$. 

\vspace{.3cm}

\begin{proof}
We are ready to compute the convergence rate of $Y_t=\mathbb{E}[\| w_{t} - w_{*} \|^2 ]$ for a given $M(t)$. We have shown that $C(t) \leq \exp(-M(t)) \int_{x=0}^{y} \exp(M(x))n(x)^2  dx + n(y)$. We are interested in the following problem: finding the largest $y\leq t$ such as 
$$\exp(-M(t))[\int_{x=0}^y \exp(M(x))n(x)^2dx] \leq n(y).$$

The solution is equal to 
$$
y = \sup \{ y\leq t: \exp(-M(t))[\int_{x=0}^y \exp(M(x))n(x)^2dx] \leq n(y) \}.
$$

Since $M(x)$ always decreases, 

\begin{align*}
y &\geq \sup \{ y\leq t: \exp(-M(t))[\int_{x=0}^y \exp(M(x))n(x) n(0) dx] \leq n(y) \}\\
  &= \sup \{ y\leq t: \exp(-M(t))n(0)[\exp(M(y))-\exp(M(0))] \leq n(y) \}\\
  &= \sup \{ y\leq t: \exp(M(y)) \leq \exp(M(0)) + \frac{n(y)}{n(0)} \exp(M(t))\}\\
  &\geq \sup \{ y\leq t: \exp(M(y)) \leq \exp(M(0)) + \frac{n(t)}{n(0)} \exp(M(t))\}\\
 &\geq \sup \{ y\leq t: \exp(M(y)) \leq  \frac{n(t)}{n(0)} \exp(M(t))\}\\  
 &=\sup \{ y\leq t: M(y) \leq  \ln [ \frac{n(t)}{n(0)} ]+ M(t)\}\\
 &= M^{-1}(\ln [ \frac{n(t)}{n(0)} ]+ M(t)),
\end{align*}
where $M^{-1}(t)$ exists for $t\in (0,n(0)]$ (since $M(y)$ strictly increases and maps into $(0,n(0)]$ for $y\geq 0$).

Therefore, $$C(t) \leq 2n(M^{-1}(\ln [ \frac{n(t)}{n(0)} ]+ M(t)))$$
and 
$$
Y_{t+1} \leq N \exp(n(0)) 2n(M^{-1}(\ln [ \frac{n(t+1)}{n(0)} ]+ M(t+1)))  + \exp(-M(t+1)) [\exp(M(1))n(0)^2 N + Y_0].
$$
\end{proof}



\noindent
\textbf{Theorem~\ref{theorem:fastest_convergence}}
\textit{Among all stepsizes $\eta_{q,t}=1/(K+t)^q$ where $q>0$, $K$ is a constant such that $\eta_{q,t}\leq \frac{1}{2L}$, SGD algorithm enjoys the fastest convergence with stepsize $\eta_{1,t}=1/(2L+t)$.} 

 \vspace{.3cm}
 
\begin{proof}
In~\eqref{eq:Yt_Ct_nt} we have 
$$
\mathbb{E}[\| w_{t} - w_{*} \|^2 ] \leq A C(t) + B \exp(-M(t)),  
$$
where $A= N \exp(n(0))$ and $B=\exp(M(1))n(0)^2N  + \mathbb{E}[\| w_{0} - w_{*} \|^2 ]$. Let us denote $C_q(t)=C(t),n_q(t)=\mu/(K+t)^q$ where $\eta_{q,t}=1/(K+t)^q$. It is obvious that $n_{q}(t)>n_1(t)$ for all $t$ and $q<1$. It implies for any $n_q(t)$ with $q<1$,  $$\exp(-M(t))<\exp(-\int_{x=0}^t\mu 1/(K+x)dx)<\exp(-\int_{x=0}^t\mu 1/x dx) < 1/t.$$
Therefore, we always have $\exp(-M(t))<1/t=n_1(t)<n_q(t)<C_q(t)$. Now, we consider the following case. We find $n(t)$ such that $n(t) = C(t)/2$. We rewrite this as 
$$ 
C(t) = \exp(-M(t)) \int_{x=0}^t \exp(M(x)) n(x)^2 dx = 2 n(t).
$$
Taking derivatives of both sides, we have:
$$
-n^2 = n(n-2n) = n(n-C) = \frac{d}{dt}C = 2 \frac{d}{dt}n .
$$ 
This is solved for $1/(at): -1/(a^2 t^2) = -2/(at^2)$ Hence, $a=1/2$ and $n(t)=2/t$. It means, $C_q(t)>C_1(t)$ and thus, the stepsize $\eta_{1,t}=1/(K+t)$ enjoys the fastest convergence. 
\end{proof}

\subsection{Convergence of Large Stepsizes in Batch Mode}
\label{subsec:convergence_large_stepsize_batch_model}
We first derive a couple lemmas which will help us deriving our main bounds. In what follows let Assumptions \ref{ass_stronglyconvex}, \ref{ass_smooth} and \ref{ass_convex} hold for all lemmas.

\begin{lem}
\label{lemma:general_form111}
Let us define $f(w;(\xi_1,\dotsc,\xi_{k}))=\frac{1}{k}\sum_{i=1}^k f(w;\xi_i)$, then we have the following properties:
$$\mathbb{E}[f(w;(\xi_1,\dotsc,\xi_{k}))]= F(w),$$
$$\mathbb{E}[\| \nabla f(w_*;(\xi_1,\dotsc,\xi_{k}))\|^2] = \frac{\mathbb{E}[\| \nabla f(w_*;\xi)\|^2]}{k},$$
and
$$
\mathbb{E}[\| \nabla f(w;(\xi_1,\dotsc,\xi_{k}))\|^2] \leq 4L [F(w)-F(w_*)] + \frac{N}{k}
$$
\end{lem}

\begin{proof}
The expectation of $f(w;(\xi_1,\dotsc,\xi_{k}))$ is equal to 
\begin{align*}
\mathbb{E}[f(w;(\xi_1,\dotsc,\xi_{k}))]&= \frac{1}{k}\sum_{i=1}^k \mathbb{E}[f(w;\xi_i)]= F(w). \tagthis \label{eq:111xx}
\end{align*}

Now we write $\mathbb{E}[\| \nabla f(w_*;(\xi_1,\dotsc,\xi_{k}))\|^2]$ as $\mathbb{E}[\sum_{j=1}^d(\frac{1}{k} \sum_{i=1}^k[\nabla f(w_*;\xi_i)]_j)^2]$. This is equal to 
\begin{align*}
 &\mathbb{E}[\sum_{j=1}^d\{ \frac{1}{k^2} \sum_{i=1}^k[\nabla f(w_*;\xi_i)]^2_j + \frac{1}{k} \sum_{i_0 \neq i_1}[\nabla f(w_*;\xi_{i_0})]_j [\nabla f(w_*;\xi_{i_1})]_j \}]\\
 =& \mathbb{E}[\sum_{j=1}^d \frac{1}{k^2} \sum_{i=1}^k[\nabla f(w_*;\xi_i)]^2_j] + \mathbb{E}[\sum_{j=1}^d \frac{1}{k} \sum_{i_0 \neq i_1}[\nabla f(w_*;\xi_{i_0})]_j [\nabla f(w_*;\xi_{i_1})]_j].
\end{align*}

The first term $\mathbb{E}[\sum_{j=1}^d \frac{1}{k^2} \sum_{i=1}^k[\nabla f(w_*;\xi_i)]^2_j]$ is equal to 
\begin{align*}
& \sum_{j=1}^d \frac{1}{k^2} \sum_{i=1}^k \mathbb{E}[[\nabla f(w_*;\xi_i)]^2_j]=
\frac{1}{k^2} \sum_{i=1}^k \mathbb{E}[\sum_{j=1}^d[\nabla f(w_*;\xi_i)]^2_j]=
\frac{1}{k^2} \sum_{i=1}^k \mathbb{E}[\|\nabla f(w_*;\xi_i) \|^2] \\
&= \frac{1}{k}E[\| \nabla f(w_*;\xi)\|^2].
\end{align*}

The second term $\mathbb{E}[\sum_{j=1}^d \frac{1}{k} \sum_{i_0 \neq i_1}[\nabla f(w_*;\xi_{i_0})]_j [\nabla f(w_*;\xi_{i_1})]_j]$ is equal to 
\begin{align*}
&\sum_{j=1}^d \frac{1}{k} \sum_{i_0 \neq i_1} \mathbb{E}[[\nabla f(w_*;\xi_{i_0})]_j] \cdot \mathbb{E}[[\nabla f(w_*;\xi_{i_1})]_j] \\ & \qquad = \sum_{j=1}^d \frac{1}{k} \sum_{i_0 \neq i_1}[ \mathbb{E}[\nabla f(w_*;\xi_{i_0})]]_j \cdot [ \mathbb{E}[\nabla f(w_*;\xi_{i_1})]]_j.
\end{align*}
Note that $\mathbb{E}[\nabla f(w_*;\xi_{i_0})]= \nabla \mathbb{E}[f(w_*;\xi_i)]=\nabla F(w_*)=0$. This means that the second term is equal to $0$ and we conclude
\begin{equation}
\label{eq:ffff11}
\mathbb{E}[\| \nabla f(w_*;(\xi_1,\dotsc,\xi_{k}))\|^2] = \frac{\mathbb{E}[\| \nabla f(w_*;\xi)\|^2]}{k}.
\end{equation}

We have the following fact:
\begin{align*}
  \| \nabla f(w;(\xi_1,\dotsc,\xi_{k})) - \nabla f(w_*;(\xi_1,\dotsc,\xi_{k}))\|^2 & 
  = \frac{1}{k^2} [\|\sum_{i=1}^k (\nabla f(w;\xi_i) - \nabla f(w_*;\xi_i))\|^2] \\  
  &\leq \frac{1}{k}  \sum_{i=1}^k \| \nabla f(w;\xi_i) - \nabla f(w_*;\xi_i)\|^2\\
&= \| \nabla f(w;\xi_i) - \nabla f(w_*;\xi_i)\|^2.
\end{align*} 
Since $\mathbb{E}[\|\nabla f(w;\xi) -\nabla f(w_*;\xi) \|^2] \leq 2L[F(w)-F(w_*)]$ (see \eqref{eq:001}), we obtain
$$
\mathbb{E}[\| \nabla f(w;(\xi_1,\dotsc,\xi_{k})) - \nabla f(w_*;(\xi_1,\dotsc,\xi_{k}))\|^2] \leq 2L[F(w)-F(w_*)]. 
$$ 

By using a similar argument as in Lemma~\ref{lem_bounded_secondmoment_04} we can derive 
\begin{align*}
\mathbb{E}[\| \nabla f(w;(\xi_1,\dotsc,\xi_{k}))\|^2] &\leq 4L [F(w)-F(w_*)] + 2\mathbb{E}[\| \nabla f(w_*;(\xi_1,\dotsc,\xi_{k}))\|^2] \\
&\overset{\eqref{eq:ffff11}}{\leq } 4L [F(w)-F(w_*)] + 2\frac{\mathbb{E}[\| \nabla f(w_*;\xi)\|^2]}{k}\\
&=  4L [F(w)-F(w_*)] + \frac{N}{k}, \tagthis \label{eq:ffffkkkk}
\end{align*}
where $N=2\mathbb{E}[\| \nabla f(w_*;\xi)\|^2]$.
\end{proof}

We define 
$$\mathcal{F}_t=\sigma(w_0,\xi'_0,u_0,\dotsc, \xi'_{t-1},u_{t-1}), $$
where
$$\xi'_i = (\xi_{i,1}, \ldots, \xi_{i,k_i}).$$
We consider the following general algorithm with the following gradient updating rule:
\begin{equation*}
\label{eq:general_updating}
w_{t+1} = w_t -\eta_t d_{\xi'_{t}} S^{\xi'_t}_{u_t} \nabla f(w_t;\xi'_t), 
\end{equation*} 
where $f(w_t;\xi'_t) = \frac{1}{k_t} \sum_{i=1}^{k_t} f(w_t;\xi_{t,i})$.
 
 \begin{lem} \label{lem:expect}  
 We have
 $$\mathbb{E}[\| d_{\xi'_t}S^{\xi'_t}_{u_t} \nabla  f(w_t;\xi'_t)  \|^2| \mathcal{F}_t, \xi'_t]\leq  D \|\nabla f(w_t;\xi'_t) \|^2$$
 and
 $$\mathbb{E}[d_{\xi'_t}S^{\xi'_t}_{u_t} \nabla  f(w_t;\xi'_t) | \mathcal{F}_t ] 
= \nabla F(w_t).$$
 \end{lem}
 
 \begin{proof}
 For the first bound, if we take the expectation of $\| d_{\xi'_t}S^{\xi'_t}_{u_t} \nabla  f(w_t;\xi'_t)  \|^2$ with respect to $u_t$, then we have (for vectors $x$ we denote the value of its $i$-th position by $[x]_i$)
 \begin{align*}
\mathbb{E}[\|d_{\xi'_t}S^{\xi'_t}_{u_t} \nabla  f(w_t;\xi'_t)  \|^2| \mathcal{F}_t, \xi'_t] & = 
 d_{\xi'_t}^2 \sum_u p_{\xi'_t}(u)  \|S^{\xi'_t}_{u} \nabla  f(w_t;\xi'_t)  \|^2 \\ & = 
 d_{\xi'_t}^2 \sum_u p_{\xi'_t}(u)  \sum_{i\in S^{\xi'_t}_{u}} [\nabla f(w_t;\xi'_t)]_i^2 \\
&= 
 d_{\xi'_t} \sum_{i\in D_{\xi'_t}} [\nabla f(w_t;\xi'_t)]_i^2 =
 d_{\xi'_t} \| f(w_t;\xi'_t) \|^2 \leq  D \|\nabla f(w_t;\xi'_t) \|^2,
\end{align*}
where the transition to the second line follows from (\ref{eq:Sexp}).

For the second bound, if we take the expectation of $ d_{\xi'_t}S^{\xi'_t}_{u_t} \nabla  f(w_t;\xi'_t)$ wrt $u_t$, then we have:
\begin{align*}
\mathbb{E}[ d_{\xi'_t}S^{\xi'_t}_{u_t} \nabla  f(w_t;\xi'_t) | \mathcal{F}_t, \xi'_t] &=
 d_{\xi'_t} \sum_u p_{\xi'_t}(u) S^{\xi'_t}_{u}  \nabla f(w_t;\xi'_t) = \nabla f(w_t;\xi'_t),
\end{align*}
and this can be used to derive
\begin{align*}
\mathbb{E}[ d_{\xi'_t}S^{\xi'_t}_{u_t}  f(w_t;\xi'_t) | \mathcal{F}_t]
=
\mathbb{E}[\mathbb{E}[d_{\xi'_t}S^{\xi'_t}_{u_t}  f(w_t;\xi'_t) | \mathcal{F}_t, \xi'_t] | \mathcal{F}_t]
= \mathbb{E}[\nabla f(w_t;\xi'_t)]  &=  \nabla F(w_t).  
\end{align*}
The last equality comes from \eqref{eq:111xx}.
\end{proof}

\begin{lem}
Let Assumptions \ref{ass_stronglyconvex}, \ref{ass_smooth} and \ref{ass_convex} hold, $0< \eta_t \leq \frac{1}{2LD}$ for all $t\geq 0$. Then, 
$$
\mathbb{E}[\|w_{t+1}-w_* \|^2] \leq (1-\mu \eta_t) \|w_t-w_*\|^2 + \eta^2_t \frac{ND}{k_t}. 
$$
\end{lem}
\begin{proof}
Since $w_{t+1} = w_t - \eta_t  d_{\xi'_t} S^{\xi'_t}_{u_t} \nabla f(w_t;\xi'_t)$, we have 
\begin{align*}
 \|w_{t+1}-w_{*}\|^2 &= \|w_t-w_{*}\|^2 - 2\eta_t \langle  d_{\xi'_t} S^{\xi'_t}_{u_t} \nabla f(w_t;\xi'_t), (w_t-w_{*})\rangle +   \eta_t^2 \|d_{\xi'_t} S^{\xi'_t}_{u_t} \nabla f(w_t;\xi'_t)\|^2.
\end{align*}

We now take expectations over $u_t$ and $\xi_t$ and use Lemmas~\ref{lem:expect} and ~\ref{lemma:general_form111}:
\begin{align*}
&\mathbb{E}[\|w_{t+1} -w_* \|^2 | \mathcal{F}_t] \\
&\leq \|w_t-w_* \|^2 - 2\eta_t \langle \nabla F(w), w_t - w_*\rangle + \eta^2_t D \mathbb{E}[\| \nabla f(w_t;\xi'_t)\|^2|\mathcal{F}_t]
\end{align*}

By \eqref{ass_stronglyconvex}, we have  
\begin{equation}
\label{eq:ppppppp}
 -  \langle \nabla F(w), w_t - w_*\rangle \leq - [F(w)-F(w_*)] - \mu/2 \| w_t - w_*\|^2
\end{equation}
Thus, $\mathbb{E}[\|w_{t+1} -w_* \|^2 | \mathcal{F}_t]$ is at most
\begin{align*}
&\overset{\eqref{eq:ppppppp}}{\leq} \|w_t-w_* \|^2 - \mu \eta_t \| w_t - w_*\|^2 - 2\eta_t[F(w)-F(w_*)] + \eta^2_t D \mathbb{E}[\| \nabla f(w_t;\xi'_t)\|^2|\mathcal{F}_t]. 
\end{align*}

Since $\mathbb{E}[\| \nabla f(w_t;\xi'_t)\|^2|\mathcal{F}_t] \leq 4L [F(w)-F(w_*)] + \frac{N}{k_t}$  (see \eqref{eq:ffffkkkk}), $\mathbb{E}[\|w_{t+1} -w_* \|^2 | \mathcal{F}_t]$ is at most 
\begin{align*}
&\overset{\eqref{eq:ffffkkkk}}{\leq} (1- \mu \eta_t) \| w_t - w_*\|^2 - 2\eta_t(1-2\eta_t LD)[F(w)-F(w_*)] + \eta^2_t \frac{ND}{k_t}. 
\end{align*}
Using the condition $\eta_t \leq \frac{1}{2LD}$ yields the lemma. 
\end{proof}

As shown above, 
$$\mathbb{E}[\| w_{t+1} - w_{*} \|^2 ]  \leq (1-\mu \eta_t) \mathbb{E}[ \| w_{t} - w_{*} \|^2 ] +  \frac{\eta_t^2 ND}{k_t},$$
when $\eta_t \leq \frac{1}{2LD}$.

Let $Y_{t+1}=\mathbb{E}[\| w_{t+1} - w_{*} \|^2 ]$, $Y_{t}=\mathbb{E}[\| w_{t} - w_{*} \|^2 ]$, $\beta_t= 1 - \mu \eta_t$ and $\gamma_t=\frac{\eta^2_tND}{k_t}$. As proved in Lemma~\ref{lemma:hogwild_recursive_form}, if $Y_{t+1} \leq \beta_t Y_t + \gamma_t$, then 
\begin{align*}
Y_{t+1} &\leq \sum_{i=0}^t [\prod_{j=i+1}^t \beta_j] \gamma_i + (\prod_{i=0}^t\beta_i)Y_0\\
&=  \sum_{i=0}^t [\prod_{j=i+1}^t (1 - \mu \eta_j)] \gamma_i + [\prod_{i=0}^t(1 - \mu \eta_i)]\mathbb{E}[\| w_{0} - w_{*} \|^2 ]\\
\end{align*}

Let us define $n(j)=\mu n_j$ and $M(y) = \int_{x=0}^y n(x) dx$ as in Section~\ref{sec:large_stepsize_convergence}. 


\vspace{.3cm}

\noindent
\textbf{Theorem~\ref{theo:asfasasda}}
\textit{Let Assumptions \ref{ass_stronglyconvex}, \ref{ass_smooth} and \ref{ass_convex} hold, $\{\eta_t\}$ is a diminishing sequence with conditions $\sum_{t=0}^{\infty} \eta_t \rightarrow \infty$ and $0< \eta_t \leq \frac{1}{2LD}$ for all $t\geq 0$. Then, the sequence $\{w_t\}$ converges to $w_*$ where}
$$w_{t+1} = w_t -\eta_t d_{\xi'_{t}} S^{\xi'_t}_{u_t} \nabla f(w_t;\xi'_t).$$ 

\vspace{.3cm}

\begin{proof}
To prove the convergence of $w_t$, we only need to prove the convergence of $$C(t) = \exp(-M(t)) \int_{x=0}^t \exp(M(x)) \frac{n(x)^2}{k(x)}dx.$$

Let $T$ denote the total number of gradient computations and define $K(t)= \int_{x=1}^t k(x)dx=T$; we have $t=K^{-1}(T)$ and $\frac{dK(x)}{dx}=k(x)$. We define $y=K(x)$ or $x=K^{-1}(y)$ with $dy=k(x)dx$. We write 
\begin{align*}
C(t) &= \exp(-M(t)) \int_{x=0}^t \exp(M(x)) \frac{n(x)^2}{k(x)}dx \\
&= \exp(-M(K^{-1}(T))) \int_{K(0)}^{K^{-1}(T)}  \exp(M(K^{-1}(y))) \frac{n(x)^2}{k(x)^2} k(x)dx \\
&\leq  \exp(-M(K^{-1}(T))) \int_{1}^{K^{-1}(T)}  \exp(M(K^{-1}(y))) \frac{n(x)^2}{k(x)^2} k(x)dx.
\end{align*}
The last inequality is based on the fact that $K(0) \geq 1$. 

Let us define $n'(x)=\frac{n(x)}{k(x)}$ and using the fact that $dy=k(x)dx$, we obtain
\begin{align*}
C(t) &= C(K^{-1}(T)) \\&\leq \exp(-M(K^{-1}(T))) \int_{1}^{K^{-1}(T)}  \exp(M(K^{-1}(y))) [n'(K^{-1}(y))]^2 dy. 
\end{align*} 

Since $K^{-1}(y)=x$, we have $\frac{dK^{-1}(y)}{dy}=\frac{1}{k(x)}$ where $y=K(x)$. This implies $\frac{dM(K^{-1}(y))}{dy}= \frac{n(K^{-1}(y))}{k(K^{-1}(y))}=n'(K^{-1}(y))$. Hence, by denoting
\begin{align*}
C'(t) &=C(K^{-1}(t)),\\
n'(x) &= \frac{n(x)}{k(x)},\\
M'(x) &= M(K^{-1}(x)),
\end{align*}
we can convert the general problem into the problem of Section~\ref{subsec:convergence_large_stepsize}. 
This implies that the analysis of $C(t)$ in Section~\ref{subsec:convergence_large_stepsize} can directly apply to analyze $C(K^{-1}(T))$. Since we already proved the convergence of $C(t)$ in Section~\ref{subsec:convergence_large_stepsize}, we obtain the theorem.   
\end{proof}

\end{document}